\RequirePackage{fix-cm}
\documentclass[smallextended]{svjour3}       
\smartqed  
\usepackage{graphicx}
\usepackage{amsmath}
\usepackage{amssymb}
\usepackage{natbib}
\allowdisplaybreaks

\newcommand{\OO}{\mathcal{O}}
\newcommand{\R}{\mathbb{R}}
\newcommand{\N}{\mathbb{N}}
\newcommand{\Z}{\mathbb{Z}}
\newcommand{\Q}{\mathbb{Q}}

\newcommand{\E}{\mathbb{E}}
\newcommand{\Prob}{\mathbb{P}}

\newcommand{\BB}{\mathcal{B}}
\newcommand{\CC}{\mathcal{C}}
\newcommand{\FF}{\mathcal{F}}
\newcommand{\GG}{\mathcal{G}}

\renewcommand{\AA}{\mathcal{A}}

\newcommand{\bn}{\mathbf{n}}
\newcommand{\bt}{\mathbf{t}}
\newcommand{\bz}{\mathbf{z}}
\newcommand{\by}{\mathbf{y}}
\newcommand{\bx}{\mathbf{x}}

\newcommand{\sd}{\,\mathrm{d}}
\newcommand{\rd}{\mathrm{d}}

\begin{document}

\title{Conditional Sampling for Max-Stable Processes with a Mixed Moving Maxima Representation}

\titlerunning{Conditional Sampling for Mixed Moving Maxima Processes}        

\author{Marco Oesting \and Martin Schlather}


\institute{M. Oesting \and  M. Schlather\at
           Institute of Mathematics, University of Mannheim\\
           A5, 6, 68131 Mannheim, Germany\\
           Tel.: +49-621-181 2563, Fax: +49-621-181 2539\\
           \email{oesting@math.uni-mannheim.de}
}

\date{}

\maketitle

\begin{abstract}
This paper deals with the question of conditional sampling and prediction for
the class of stationary max-stable processes which allow for a mixed moving
maxima representation. We develop an exact procedure for conditional sampling
using the Poisson point process structure of such processes. For explicit
calculations we restrict ourselves to the one-dimensional case and use a
finite number of shape functions satisfying some regularity conditions. For
more general shape functions approximation techniques are presented. Our 
algorithm is applied to the Smith process and the Brown-Resnick process.
Finally, we compare our computational results to other approaches. Here,
the algorithm for Gaussian processes with transformed marginals turns out
to be surprisingly competitive.

\end{abstract}

\section{Introduction} \label{sec-intro}

Over the last decades, several models for max-stable processes have been 
developed and applied. In view of the wide range of potential applications of 
max-stable processes for modelling extreme events, the question of prediction
and conditional sampling arises. \citet{davis-resnick-1989, davis-resnick-1993}
proposed prediction procedures for time series which basically aim to minimize 
a suitable distance between observation and prediction. Further approaches for
max-stable processes have been rare for a long time, apart from a few 
exceptions. \cite{cooley-etal-2012} introduced an approximation of the 
conditional density. Recently, \cite{wang-stoev-2010} proposed an exact and 
efficient algorithm for conditional sampling for max-linear models 
$$ Z_i = \max_{j=1,\ldots,p} a_{ij} Y_j, \quad i=1,\ldots,n,$$
where $Y_j$ are independent Fr\'echet random variables. 
\cite{dombry-ribatet-2012} presented algorithms for conditional simulation of
Brown-Resnick processes and extremal Gaussian processes based on more general
results on conditional distributions of max-stable processes given in
\cite{dombry-2011}.
\medskip

Here, we consider stationary max-stable processes with standard Fr\'echet 
margins that allow for a mixed moving maxima (M3) representation (see, for 
instance, \citealp{schlather-2002}, \citealp{stoev-taqqu-2006}). Let $G$ be a
countable set of measurable functions $f: \ \R^d \to [0,\infty)$ and 
$\GG = 2^G$. Furthermore, let $(\Omega, \FF, \Prob)$ be a probability space and
$F: (\Omega, \FF) \to (G,\GG)$ be a random function such that
$\E(\int_{\R^d} F(x) \sd x) = 1$.
Then, we consider the stationary max-stable process
\begin{equation}
 Z(t) = \max_{(s,u,f) \in \Pi} u f(t-s), \qquad t \in \R^d, \label{eq:procdef}
\end{equation}
where $\Pi$ is a Poisson point process on $S = \R^d \times (0,\infty) \times G$
with intensity
\begin{equation}
\Lambda(A \times B \times C) = \mu(A) \cdot \Prob_F(C) \cdot \int_B \frac {\sd u} {u^2},
\quad A \in \BB^d, \ B \in \BB \cap (0,\infty), \ C \in \GG, \label{eq:intensity}
\end{equation}
$\mu$ is the Lebesgue measure on $\R^d$ and $\Prob_F$ the push forward measure
of $F$ on $G$. \citet{stoev-taqqu-2006} provide the equivalent representation
of M3
\begin{equation} \label{eq:extremal-integral}
Z(t) =  \bigvee_{f \in G} \Prob_F(\{f\}) \cdot
  \int_{\R^d}^{\!\!\!\!\!\!e} f(t-u) M_{1}^{(f)}({\rm d} u), \quad t \in \R^d,
\end{equation}
as an extremal integral where $M_{1}^{(f)}$, $f \in G$, are independent copies
of a random sup-measure $M_1$ on $\R^d$ w.r.t.\ $\mu$ 
\citep[cf.][Def.\ 2.1]{stoev-taqqu-2006}.

We aim to sample from the conditional distribution of the process $Z$ given 
$Z(t_1), \ldots, Z(t_n)$ for fixed $t_1, \ldots, t_n \in \R^d$. As $Z$ is 
entirely determined by the Poisson point process $\Pi$, we analyse the 
distribution of $\Pi$ given some values of $Z$. The idea to use a Poisson point
process structure for calculating conditional distributions has already been
implemented in the case of a bivariate min-stable random vector 
\citep{weintraub-1991}.

A very general Poisson point process approach was recently used by 
\cite{dombry-2011}. They separately consider the points of the Poisson point
process which contribute to the maximum process $Z$ in $t_1, \ldots, t_n$ and
those which do not. They provide formulae for the distribution of these two 
point processes in terms of the exponent measure. Via these formulae, the 
resulting conditional distribution function can be calculated explicitly if the
exponent measure  is absolutely continuous w.r.t.\ the Lebesgue measure as in
the case of Brown-Resnick and extremal Gaussian processes 
\citep[cf.][]{dombry-ribatet-2012}. However, in case of a non-regular model, 
like M3 with a countable number of shape functions,
the formulae cannot be directly applied for explicit computations. Therefore,
we will use a different approach, based on martingale arguments leading to
explicit formulae. As we also consider the points contributing to the maximum
separately, some of the results of \cite{dombry-2011} are independently 
established here.
\medskip

As an example for the Poisson point process approach, we consider the case of
two observations $Z(t_1) = z_1$ and $Z(t_2) = z_2$.
Then, by definition of $Z$ there is at least one point $(s_1, u_1, f_1) \in \Pi$
that generates $Z(t_1)$, i.e.\ $u_1 f_1(t_1-s_1) = z_1$, and at least one point
$(s_2, u_2, f_2) \in \Pi$ with $u_2 f_2(t_2-s_2) = z_2$. Later, we will show
that each observation is generated by exactly one point. Thus, there are two
different possible point configurations which we will call scenarios, similarly
to \cite{wang-stoev-2010} and \cite{dombry-2011}: (i) a single point generates
both observations, i.e.\ $(s_1,u_1,f_1) = (s_2,u_2,f_2)$, and (ii) the points
$(s_1,u_1,f_1)$ and $(s_2,u_2,f_2)$ are different.
Then, conditional sampling of $\Pi$ can be  performed via the following steps. 
First, draw a scenario from the conditional scenario distribution. Then,
within this scenario, simulate the points generating the observations.
Finally, independently simulate those points of $\Pi$ that do not generate any
observation.
\medskip

The paper is organized as follows. In Section \ref{sec-measure}, we introduce a
random partition of $\Pi$ into three measurable point processes allowing
to focus on those points of $\Pi$ which determine $Z(t_1), \ldots, Z(t_n)$.
We figure out the conditional distribution of the resulting scenarios coping
with the problem that the condition $\{Z(t_1) = z_1, \ldots, Z(t_n)=z_n\}$ is 
an event of probability zero for every $z_1,\ldots,z_n > 0$ 
(Section \ref{sec-limits}). Based on these considerations, Section 
\ref{sec-finite} provides explicit formulae for the conditional distribution of
$\Pi$ for the case $d=1$ and some regularity assumptions on a finite number of
random shape functions. In Section \ref{sec-compare}, the results are applied
to Smith's \citeyearpar{smith-1990} process, whose shape function is the 
Gaussian pdf, and compared to other algorithms. Section \ref{sec-approx} deals
with an approximation procedure in the case of a countable and uncountable 
number of random shape functions. A prominent example, the Brown-Resnick
process \citep{brown-resnick-1977}, is further investigated in a comparison
study for different algorithms in Section \ref{sec-BR}. In Section 
\ref{sec-discrete}, we give a brief overview of the results for a discrete
M3 process restricted to $p\Z^d$.
The results from theoretical considerations as well as from the simulation
studies are summarized and discussed in Section \ref{sec-discussion}.
Finally, Section \ref{sec-calculate} provides the proofs for the results in
Section \ref{sec-finite}.

\section{Random partition of $\Pi$ and measurability} \label{sec-measure}

In this section, we will consider random sets of points within $\Pi$ which 
essentially determine the process $Z$. Separating these critical points of 
$\Pi$ from the other ones, we get a random partition of $\Pi$. We will show
that this partition is measurable, which allows for further investigation of
this partition.
\medskip

For some fixed $(t,z) \in \R^d \times (0,\infty)$, define the set
$$ K_{t,z} ={} \Big\{ (x,y,f) \in S: \ y = \frac{z}{f(t-x)} \Big\}, $$
where we use the convention $z/0 = \infty$ for $z>0$. We call $K_{t,z}$ the 
\emph{set of points generating $(t,z)$} due to the fact that
$$ Z(t) = z \quad \Longleftrightarrow \quad |\Pi \cap K_{t,z}| \geq 1 \ \wedge \
 \Pi \cap \overline{K_{t,z}}=\emptyset.$$
Here,
$\overline{K} = \bigcup_{(x,y,f)\in K} \{x\} \times (y,\infty) \times \{f\}$
for a set  $K \subset S$ and $\wedge$ denotes the logical conjunction `and'.
\medskip

Hereinafter, for any mapping $g$ with domain ${\rm dom}(g) \subset \R^d$ and 
any vector $\bt =(t_1, \ldots,t_n) \in ({\rm dom} g)^n$  we will write $g(\bt)$
instead of $(g(t_1), \ldots, g(t_n))$, for short. Similarly, $\bt > 0$ is
understood as $t_i > 0$, $i=1,\ldots,n$.

We now consider $n$ fixed points 
$(t_1,z_1),\ldots,(t_n,z_n) \in \R^d \times (0,\infty)$ and the 
\emph{set of points generating $(\bt, \bz)$} as
\begin{align*}
K_{\bt,\bz}
 & ={} \bigg\{(x,y,f) \in S: \ y = \min_{i=1,\ldots,n} \frac{z_i}{f(t_i-x)} \bigg\}\\
={} & \big\{(x,y,f) \in S: \ y f(\bt-x) \leq \bz, 
    y f(t_j-x) = z_j \textrm{ for some } j \in \{1,\ldots,n\}\big\}.
\end{align*}
This implies
\begin{align*}
\overline{K_{\bt,\bz}} ={} & \big\{(x,y,f) \in S: \ yf(t_j-x) > z_j \textrm{ for some } 
j \in \{1,\ldots,n\} \big\}\\
\textrm{and } K_{\bt,\bz} \cap K_{t_i,z_i} ={} &
    \left\{(x,y,f) \in S: \ y f(t_i-x) = z_i,\ y f(\bt-x) \leq \bz \right\}.
\end{align*}
Therefore, we have that $Z(\bt) \leq \bz$ if and only if 
$\Pi \cap \overline{K_{\bt,\bz}} = \emptyset$ and
\begin{align}  Z(\bt)=\bz \hfill
 \iff{} & \hfill |\Pi \cap K_{t_i,z_i} \cap K_{\bt,\bz}| \geq 1, \ i=1,\ldots,n \ 
         \wedge \ \Pi \cap \overline{K_{\bt,\bz}} = \emptyset. \label{eq:equivalence}
\end{align}

Now we define a random partition of $\Pi$ by
\begin{align*}
\Pi_1 :={}  \Pi \cap \overline{K_{\bt,Z(\bt)}}, \quad
\Pi_2 :={} \Pi \cap K_{\bt,Z(\bt)},\quad
\text{and } \Pi_3 :={} \Pi \setminus (\Pi_1 \cup \Pi_2).
\end{align*}
Relation \eqref{eq:equivalence} implies that
$\Pi_1 = \emptyset$ and $|\Pi_2 \cap K_{t_i,Z(t_i)}| \geq 1$ a.s. for
$i \in \{1,\ldots,n\}$. Note that the processes $\Pi_2$ and $\Pi_3$ can be
transformed into the processes $\Phi_K^+$ and $\Phi_K^-$ defined in 
\cite{dombry-2011} via the transformation $(x,y,f) \mapsto yf(\cdot-x)$. 
We need a refined version of a result by \cite{dombry-2011} who proved the 
measurability of $\Phi_K^+$ and $\Phi_K^-$ in a more general setting. Here, we
will show the measurability of a further partition of $\Pi_2$, namely the
restriction of $\Pi_2$ to certain intersection sets, which we will need in
Section \ref{sec-limits}.

For any $A \in \AA$, $\AA = 2^{\{1,\ldots,n\}}\setminus\{\emptyset\}$, we 
define
\begin{align*}
 I_A(\bz) ={} & \textstyle K_{\bt,\bz} \cap \big(\bigcap_{i \in A} K_{t_i,z_i} 
\setminus \bigcup_{j \in A^c} K_{t_j,z_j}\big)\\
={} & \{(x,y,f) \in S: \ y f(t_i-x) = z_i, \ i \in A, \
       y f(t_j - x) < z_j, \ j \notin A\},
\end{align*}
i.e.\ $I_A(\bz)$ contains those points which simultaneously generate all the
observations $(t_i, z_i)$, $i \in A$, but none of the observations 
$(t_j, z_j)$, $j \notin A$.
By construction $K_{\bt,\bz}$ is a disjoint union of $I_A(\bz)$, $A \in \AA$.

To prove the measurability of these restrictions of $\Pi$ to the intersection
sets $I_A(Z(\bt))$, let $\CC$ be the $\sigma$-algebra on $\R^{\R^d}$ generated
by the cylinder sets
$$C_{s_1, \ldots, s_m}(B) = \{f \in \R^{\R^d}: \ (f(s_1),\ldots,f(s_m)) \in B\},$$
where $s_1,\ldots,s_m \in \R^d, B \in  \BB^m$, $m \in \N$.

\begin{proposition} \label{measurability}
Let $t_1,\ldots,t_n \in \R^d$ be fixed.
\begin{enumerate}
\item The mapping 
$$ \Psi: \ S \to \R^{\R^d}, \ (x,y,f) \mapsto y f(\cdot-x)$$
 is $(\BB^d \times(\BB \cap (0,\infty)) \times 2^G, \CC)$-measurable.
\item Let $A \in \AA$ and $B \subset S$ a bounded Borel set.
 Then, $|\Pi \cap I_A(Z(\bt)) \cap B|$ is a random variable.
\item $\Pi_1, \Pi_2$ and $\Pi_3$ are point processes.
 \citep[cf.][]{dombry-2011}
\end{enumerate}
\end{proposition}
\begin{proof}
 \begin{enumerate}
  \item It suffices to verify that 
    $\Psi^{-1}(C_{s_1,\ldots,s_m}(\times_{i=1}^m (a_i,b_i)))$ is measurable for
    any $s_j \in \R^d$, $a_j < b_j \in \R$, $j=1,\ldots,m$, $m \in \N$. We have
    \begin{align*}
         & \Psi^{-1}(C_{s_1,\ldots,s_m}(\times_{i=1}^m (a_i,b_i)))\\
         ={} & \bigcup_{f \in G} \bigcup_{\substack{t \in \R^d}} \{t\} \times 
             \bigg( \bigvee_{i=1,\ldots,m} \frac{a_i}{f(s_i-t)}, \bigwedge_{i=1,\ldots,m} \frac{b_i}{f(s_i-t)}\bigg)
             \times \{f\}\\
         ={} & \bigcup_{f \in G} \bigcup_{\substack{q_1,q_2 \in \Q_{+}\\ q_1<q_2}} 
         \left\{t \in \R^d: \ (q_1,q_2) \subset \left( \bigvee_{i=1}^m \frac{a_i}{f(s_i-t)}, \bigwedge_{i=1}^m \frac{b_i}{f(s_i-t)}\right)\right\}\\
         & \hspace{7.75cm} \times (q_1,q_2) \times \{f\}.
    \end{align*}
    As each $f \in G$ is measurable, sets of the type 
    $\{t \in \R^d: f(s_i-t) \in B\}$ are measurable for any $B \in \BB$.
    Therefore, $\Psi^{-1}(C_{s_1,\ldots,s_m}(\times_{i=1}^m (a_i,b_i)))
    \in \BB^d \times (\BB \cap (0,\infty)) \times 2^G$. \qed
  \item We consider
    \begin{align*}
      & \{\omega: \ |\Pi \cap I_A(Z(\bt)) \cap B| =k\}\\
      ={} & \textstyle \bigcup_{n_0 \in \N} \bigcap_{m=n_0}^\infty \big(
         \bigcup_{\by \in \Q^n} \big\{\omega \in \Omega: \ Z(\bt) \in \times_{i=1}^n \left(y_i - 1 / m, y_i +  1 / m\right), \\
      & \hspace{1.05cm} 
        \big| \Pi \cap \Psi^{-1}\big(\big\{f \in \R^{\R^d}: \ f(t_i) \in \left(y_i -  1 / m, y_i + 1 / m\right), \ i \in A, \\
      & \hspace{4.4cm} \ f(t_j) \leq y_j - 1 / m, \ j \notin A \big\}\big) \cap B\big| = k \big\} \big).
    \end{align*}
    By the first part of this proposition, $\Psi$ is a measurable mapping and
    we get that $\{\omega: \ |\Pi \cap I_A(Z(\bt)) \cap B| =k\}$ is measurable.
    \qed
  \item For any bounded Borel set $B \subset S$ the second part of this 
    proposition yields that $|\Pi_1 \cap B| = 0$, $|\Pi_2 \cap B|
    = \sum_{A \in \AA} |\Pi \cap I_A(Z(\bt)) \cap B|$ and $|\Pi_3 \cap B|
    = |\Pi \cap B| - |\Pi_2 \cap B|$ are measurable. Thus, $\Pi_1$, $\Pi_2$ and
    $\Pi_3$ are point processes \citep[][Cor.\ 6.1.IV]{daley-vere-jones-1988}.
    \qed
 \end{enumerate}
\end{proof}

\section{Blurred sets, scenarios and limit considerations} \label{sec-limits}

This section mainly deals with the analysis of the distribution of the set of
critical points, $\Pi_2$. First, we note that, for every $\bz > \mathbf{0}$, 
the set $K_{\bt,\bz}$ has intensity measure zero. Therefore, conditional on
$Z(\bt) = \bz$, the distribution of $\Pi_2 = \Pi \cap K_{\bt,\bz}$ cannot be
calculated straightforward.
We need to borrow arguments from martingale theory, taking limits of 
probabilities conditional on the observations being in small intervals
containing $\bz$. By this conditioning, the set of critical points gets blurred
covering an area of positive measure. We distinguish between different
scenarios which are defined by the number of points of $\Pi$ in each 
intersection set $I_A(Z(\bt))$, $A \in \AA$, i.e.\ the number of points 
which influence the different observations. 
Using general bounds for the rate of convergence of the intensity of the 
blurred sets, we prove that each observation is generated by exactly one
point of $\Pi$ (Corollary \ref{as-onepoint}). This property restricts the 
number of scenarios that occur with positive probability. According to the
blurred sets, the intersection sets and the corresponding scenarios get 
blurred, as well.
The blurred scenarios are not exactly the same as the scenarios conditional on
blurred observations, but much more tractable. However, both events 
asymptotically yield the same conditional probability (Theorem
\ref{generalconvergence}). Based on these considerations, the independence of
$\Pi_2$ and $\Pi_3$ conditional on $Z(\bt)$ is shown (Corollary 
\ref{condindependence}). This allows to simulate $\Pi_2$ and $\Pi_3$ 
independently. Further, $\Pi_3$ turns out to be easily simulated (Corollary
\ref{condindependence}).
\medskip

Let $\FF_m = \sigma\left(\left\{Z(t_i) \in \left(2^{-m}k,2^{-m}(k+1)\right],
\ i=1,\ldots,n,\ k \in \N_0\right\}\right)$ 
where $\N_0 = \N \cup \{0\}$.
Then, $\{\FF_m\}_{m\in\N}$ is a filtration and 
$\FF_\infty := \bigcap_{m \in \N} \FF_m = \sigma(Z(\bt))$.
Furthermore, for $z >0$, let $j_m(z) \in \N_0$ be such that $z \in A_m(z)$ with
$$A_m(z) = \left(2^{-m} j_m(z), 2^{-m} (j_m(z)+1)\right].$$
Thus, we have $A_m(z) \stackrel{m\to\infty}{\longrightarrow} \{z\}$
monotonically and with 
$$A_m(\bz) = \left( 2^{-m} j_m(\bz), 2^{-m} (j_m(\bz) + 1)\right] = \times_{i=1}^n A_m(z_i),$$
we obtain $\{\omega \in \Omega: \, Z(\bt) \in A_m(\bz)\} \in \FF_m$.
Now, we apply L\'evy's ``Upward'' Theorem 
\citep[][Thm. 50.3]{rogers-williams-I}:
For a filtration $\{\FF_m\}_{m\in\N}$, the $\sigma$-algebra
$\FF_\infty = \bigcap_{m \in \N} \FF_m$ and any random variable $X$ with
$\E |X| < \infty$ we have 
$\lim_{m \to \infty} \E (X \mid \FF_m) = \E (X \mid \FF_\infty)$  a.s.
Thus, for $X = \mathbf{1}_{\Pi \in B}$ with $B \in \sigma(\Pi)$ where 
$\sigma(\Pi)$ denotes the $\sigma$-algebra generated by $\Pi$, we get
\begin{equation} \label{eq:levy2} \textstyle
 \lim_{m \to \infty} \Prob(\Pi \in B \mid Z(\bt) \in A_m(\bz)) = \Prob(\Pi \in B \mid Z(\bt)= \bz)  
\end{equation}
for $\Prob_{Z(\bt)}$-a.e.\ $\bz > \mathbf{0}$, where $\Prob_{Z(\bt)}$ is the
push forward measure of $Z(\bt)$. It can be easily seen that $\sigma(\Pi)$ can
be generated by the countable set of events
\begin{align*}
\mathcal{E} ={} & \big\{ \{ \omega \in \Omega: \ |\Pi \cap I_A(Z(\bt)) \cap([a,b] \times [c,\infty) \times \{f\})| =k\}, \\
                & \hspace{2.4cm} A \in 2^{\{1,\ldots,n\}}, \ a \leq b \in \Q^d, \ 0 < c \in \Q,
                     \ f \in G, \ k \in \N_0\big\}.
\end{align*}
Note that $A$ runs through $2^{\{1,\ldots,n\}}$ instead of $\AA$.

Let $\mathcal{E}^*$ be the set of all finite intersections within 
$\mathcal{E}$, i.e.\
$ \mathcal{E}^* = \{ \cap_{i=1}^n E_i, \ E_i \in \mathcal{E}, \ i=1, \ldots,n, \ n \in \N\}$.
As $\mathcal{E}^*$ is countable, $\Prob( \Pi \in \cdot \mid Z(\bt) = \bz)$ can
be well-defined on $\mathcal{E}^*$ by \eqref{eq:levy2} for 
$\Prob_{Z(\bt)}$-a.e.\ all $\bz > \mathbf{0}$, and thus can be extended to a 
probability measure on $\sigma(\mathcal{E}^*) = \sigma(\Pi)$.
Furthermore, by construction via the limit, the mapping 
$\bz \mapsto \Prob(\Pi \in E \mid Z(\bt) = \bz)$ is measurable on a set of 
probability one for any $E \in \sigma(\Pi)$ and the equality
$$ \textstyle \int_A \Prob(\Pi \in E \mid Z(\bt) = \bz) \Prob_{Z(\bt)}(\rd \bz) = \Prob(\Pi \in E, \ Z(\bt) \in A)$$
holds for every $A \in \sigma(Z(\bt))$, $E \in \sigma(\Pi)$. Thus, L\'evy's
``Upward'' Theorem will yield a regular conditional probability distribution.
\medskip

Let 
\begin{align*}
K_{t,z}^{(m)} =\, & \textstyle \bigcup_{\tilde z \in A_m(z)} K_{t,\tilde z}
 = \{(x,y,f) \in S: y f(t-x) \in A_m(z)\},\displaybreak[0]\\
K_{\bt, \bz}^{(m)} =\, & \textstyle \bigcup_{\tilde \bz \in A_m(\bz)} K_{\bt, \tilde \bz}
 = \big\{(x,y,f) \in S: y f(\bt-x) \leq 2^{-m} (j_m(\bz)+1),\\
  & \hspace{2.6cm} y f(t_i-x) \in A_m(z_i) \textrm{ for some } i \in \{1,\ldots,n\} \big\},\displaybreak[0] \\
\textrm{and }
\overline{K_{\bt,\bz}}^{(m)} 
=\, & \textstyle \bigcap_{\tilde \bz \in A_m(\bz)} 
      \overline{K_{\bt, \tilde \bz}} {}={} \big\{(x,y,f) \in S: y f(t_i-x) > 2^{-m} (j_m(z_i)+1)\\
     & \hspace{2.7cm} \textrm{ for some } i \in \{1,\ldots,n\} \big\} 
\,=\, \overline{K_{\bt,2^{-m} (j_m(\bz)+1)}}.
\end{align*}
See Fig.\ \ref{fig:blurred_curves} for some illustration. 
These definitions imply that
\begin{align*}
K^{(m)}_{\bt,\bz} \cap \overline{K_{\bt,\bz}}^{(m)} {}={} \emptyset
\quad \textrm{and} \quad K^{(m)}_{\bt,\bz} \cup \overline{K_{\bt,\bz}}^{(m)}
{}={} \overline{K_{\bt,2^{-m} j_m(\bz)}}.
\end{align*}
We call these sets the blurred sets belonging to $Z(\bt)$ conditional on 
$Z(\bt) \in A_m(\bz)$. This notation is due to the fact that we have 
$Z(\bt) \leq 2^{-m} (j_m(\bz)+1)$ if and only if 
$\Pi \cap \overline{K_{\bt,\bz}}^{(m)} = \emptyset$.
Furthermore, as
\begin{align*}
K_{\bt,\bz}^{(m)} \cap K^{(m)}_{t_i,z_i} ={} & \big\{(x,y,f) \in S: \ y f(t_i-x) \in A_m(z_i), \\
  & \hspace{2.45cm} y f(\bt-x) \leq 2^{-m}(j_m(\bz)+1)\big\},
\end{align*}
we get that $|\Pi \cap K^{(m)}_{t_i,z_i} \cap K^{(m)}_{\bt,\bz}| \geq 1$ for
every $i \in \{1,\ldots,n\}$ implies that $Z(\bt) > 2^{-m} j_m(\bz)$. 
Thus, we obtain that $Z(\bt) \in A_m(\bz)$ if and only if
\begin{align}
|\Pi \cap K^{(m)}_{t_i,z_i} \cap K^{(m)}_{\bt,\bz}| \geq 1, \ i=1,\ldots,n \
\wedge \ \Pi \cap \overline{K_{\bt,\bz}}^{(m)} = \emptyset. \label{eq:blurred-equivalence}
\end{align}
In particular, for fixed $\bz \in (0,\infty)^n$, the point process
$\Pi \setminus (K^{(m)}_{\bt,\bz} \cup \overline{K_{\bt,\bz}}^{(m)})$ is
independent of the event $Z(\bt) \in A_m(\bz)$.
\medskip

Based on these blurred sets, we define the
\emph{blurred intersection sets}
$$I_A^{(m)}(\bz) = \textstyle K_{\bt,\bz}^{(m)} \cap 
\left(\bigcap_{i \in A} K_{t_i,z_i}^{(m)}  
\setminus \bigcup_{j \in A^c} K_{t_j,z_j}^{(m)}\right), \qquad A \in \AA.$$
We note that $K_{\bt,\bz}^{(m)}$ can be written as a disjoint union of
$I_A^{(m)}(\bz)$, $A \in \AA$.
\medskip

\begin{lemma} \label{lambdaconv}
 For any $A \in \AA$ and $\bz > \mathbf{0}$ we have
 $\Lambda\big(I_A^{(m)}(\bz)\big) \in \OO(2^{-m})$, i.e.~ 
 $\limsup_{n\to\infty} 2^m \Lambda\big(I_A^{(m)}(\bz)\big) < \infty$, 
 where $\Lambda(\cdot)$ is given
 by \eqref{eq:intensity}.
\end{lemma}

\begin{proof}
It suffices to show that $ \Lambda\big(K_{t_i,z_i}^{(m)}\big)
\in \OO(2^{-m})$ for all $i=1,\ldots,n$. By a straightforward 
computation we get
\begin{align*}
 \Lambda\left(K_{t_i,z_i}^{(m)}\right) &
={} \E_F\bigg(\int_{\R^d} \int_{2^{-m} j_m(z_i) / F(t_i-x)}^{2^{-m} (j_m(z_i)+1) / F(t_i-x)}
 u^{-2} \sd u \sd x\bigg) \displaybreak[0]\\
={} & \E_F\bigg(\int_{\R^d} F(t_i-x) \sd x\bigg) \cdot 
\bigg(\frac {2^m} {j_m(z_i)} - \frac {2^m} {j_m(z_i)+1}\bigg)
={} \frac {\frac 1 {2^m}}{\frac {j_m(z_i)}{2^m} \frac {j_m(z_i)+1}{2^m}}.
\end{align*}
As $\lim_{m\to\infty} 2^{-m} j_m(z_i) = z_i$, the assertion of the lemma
follows. \qed
\end{proof}
\medskip

In Section \ref{sec-finite}, a more precise notion about the speed of 
convergence of $2^{-m} j_m(z_i) \to z_i$ will be useful.

\begin{lemma} \label{lambda-bound}
 For any $\varepsilon > 0$, with probability one we have
 \begin{align*}
 & \lim_{m \to \infty} 2^{m(1+\varepsilon)} \min_{i=1,\ldots,n} \left(2^{-m}(j_m(Z(t_i))+1) - Z(t_i)\right) = \infty \\
 \textrm{and} \quad &  \lim_{m \to \infty} 2^{m(1+\varepsilon)} \max_{i=1,\ldots,n} \left( 2^{-m} j_m(Z(t_i)) - Z(t_i)\right) = - \infty.
 \end{align*}
\end{lemma}
\begin{proof}
 For the first assertion it suffices to show that
$$ \liminf_{m \to \infty} 2^{m(1+\varepsilon)} \left( 2^{-m} (j_m(Z(t_i))+1) - Z(t_i)\right) = \infty, \quad i=1, \ldots,n.$$
Let $a \in \left(0, \frac 1 2\right)$.
Then, for $m \in \N$ enough, we have
\begin{align*}
  \Prob( 2^{-m} (j_m(Z(t_i))+1) -& Z(t_i) \leq  a 2^{-m}) ={} \sum_{k=1}^{\infty} \Prob\left((k-a)2^{-m} \leq Z(t_i) \leq k 2^{-m}\right)
 \displaybreak[0]\\
={} & \sum_{k=1}^\infty \exp\left(- \frac 1 {k2^{-m}}\right) \cdot \left(1 - \exp\left(-\frac 1 {k2^{-m}} \cdot \frac {a/k}{1 - a/k} \right)\right) \displaybreak[0]\\
\leq{} & \sum_{k=1}^\infty \exp\left(- \frac 1 {k2^{-m}}\right) \cdot \left(\frac 1 {k2^{-m}} \cdot 2 \frac a k\right) {} \leq 4 a,
\end{align*}
where we used the fact that $1 - \exp(-x) \leq x$ for all $x > 0$ for the first
inequality and the approximation of the Riemann integral in the second 
inequality. Thus, for $a = C 2^{-m\varepsilon}$ with $C>0$ and $m$ large 
enough, we get
\begin{align*}
 & \Prob\left( 2^{-m}(j_m(Z(t_i))+1) - Z(t_i) \leq C\cdot 2^{-m(1+\varepsilon)}\right) 
{}\leq{}  4 C \cdot (2^\varepsilon)^{-m}.
\end{align*}
Therefore, the probabilities above are summable with respect to $m$ and thus
\begin{align*}
& \Prob\left( \liminf_{m \to \infty} 2^{m(1+\varepsilon)} \left(2^{-m} (j_m(Z(t_i))+1) - Z(t_i)\right) < C\right)=0
\end{align*}
for any $C > 0$ by the Borel-Cantelli lemma. 
The second assertion can be shown analogously. \qed
\end{proof}

Now, we introduce disjoint \emph{``blurred'' scenarios}
\begin{align*}
 E^{(m)}_{\bn}(\bz) = \{\omega \in \Omega: \ & |\Pi \cap I_A^{(m)}(\bz)| = n_A, \, A \in \AA, \, |\Pi \cap \overline{K_{\bt,\bz}}^{(m)}| = 0\}
\end{align*}
with $\bn = (n_A)_{A \in \AA} \in N_1$ where 
$$N_1 = \left\{ (n_A)_{A \in \AA} \in \N_0^{2^n-1}:\ \textstyle \sum_{A:\,A \ni i} n_A \geq 1, \ i=1,\ldots,n\right\}.$$
Thus, the event $\{ Z(\bt) \in A_m(\bz)\}$ is the disjoint union 
$\bigcup_{\bn \in N_1}  E^{(m)}_{\bn}(\bz)$. In the same way, dropping the 
$(m)$ in the definition, we specify \emph{scenarios} $E_{\bn}(\bz)$.
\medskip

Now, we show that $|\Pi \cap K_{t_i,Z(t_i)}| = 1$ a.s.\ for every 
$i \in \{1,\ldots,n\}$. To this end, we first verify that, with probability
one, $\Pi_2$ does not contain any point that can be removed without any effect
on $Z(\bt)$. To this end, we consider scenarios $E_{\bn}(Z(\bt))$ with 
$\bn \in N_2$ defined by 
$$ N_2 = \{ (n_A + \mathbf{1}_{A = A^*})_{A \in \AA}: \ (n_A)_{A \in \AA} \in N_1, \ A^* \in \AA \}.$$
Then, $\bn \in N_2$ if and only if $E_{\bn}(Z(\bt))$ allows for removing at 
least one point without influencing $Z(\bt)$.

\begin{lemma} \label{only-nec-points}
For $\Prob_{Z(\bt)}$-a.e.~$\bz > \mathbf{0}$, we have
\begin{align*}
&  \textstyle \Prob\left( \bigcup_{\bn \in N_2} E_{\bn}(Z(\bt)) \, \Big| \, Z(\bt) = \bz\right) \\
={} & \lim_{m \to \infty} \textstyle \Prob\left( \bigcup_{\bn \in N_2} E^{(m)}_{\bn}(\bz) \, \Big|
\, Z(\bt) \in A_m(\bz)\right) = 0.
\end{align*}
In particular,
 $\lim_{m \to \infty} \Prob\left(|\Pi \cap K^{(m)}_{\bt,\bz} \setminus \Pi_2| > 0 \ \Big| \ Z(\bt) \in A_m(\bz)\right)
    = 0.$
\end{lemma}
\begin{proof}
 By L\'evy's ``Upward'' Theorem,
 for $\Prob_{Z(\bt)}$-a.e.\ $\bz > \mathbf{0}$, we have
\begin{align*}
 \sum_{\bn \in N_2} \Prob\left(E_{\bn}(Z(\bt)) \mid Z(\bt) = \bz\right)
{}={} \lim_{m \to \infty} \sum_{\bn \in N_2} \Prob\left(E_{\bn}(Z(\bt)) \mid Z(\bt) \in A_m(\bz)\right).
\end{align*}
 As $\bigcup_{\bn \in N_2} E_{\bn}(Z(\bt))
 \subset \bigcup_{\bn \in N_2} E^{(m)}_{\bn}(Z(\bt))$,
 it suffices to show
\begin{align}
 \lim_{m \to \infty} \sum_{\bn \in N_2} \!\! \Prob(E^{(m)}_{\bn}(\bz) \, 
            \big| \, Z(\bt) \in A_m(\bz)) 
{}={}  \lim_{m \to \infty} \frac{\sum_{\bn \in N_2}
    \Prob(E^{(m)}_{\bn}(\bz))}{\sum_{\bn \in N_1}
    \Prob(E^{(m)}_{\bn}(\bz))} = 0 \label{eq:quotient}.
\end{align}
for every $\bz > \mathbf{0}$.

We note that, for a Poisson variable $X$ with parameter $\lambda \geq 0$, we
have $\Prob(X=k+1) \leq \lambda \Prob(X=k)$ for all $k \in \N_0$. Hence, for
any $\bn \in N_1$ and $A^* \in \AA$, 
 $$\Prob(E_{(n_A + \mathbf{1}_{A = A^*})_{A \in \AA}}^{(m)}(\bz)) \leq{} \Prob(E_{\bn}^{(m)}(\bz)) \cdot \Lambda(I_{A^*}^{(m)}(\bz))$$
and hence, for fixed $m \in \N$, the left-hand side of \eqref{eq:quotient} is
bounded by
\begin{align*}
  \frac{\sum_{\bn \in N_{1}} \sum_{A^* \in \AA}
    \Prob(E^{(m)}_{(n_A + \mathbf{1}_{A=A^*})_{A \in \AA}}(\bz))}{\sum_{\bn \in N_1} \Prob(E^{(m)}_{\bn}(\bz))}
\leq{}  \sum_{A^* \in \AA} \Lambda(I_{A^*}^{(m)}(\bz)).
\end{align*}

By Lemma \ref{lambdaconv}, we get equality \eqref{eq:quotient}, i.e.{} the
first assertion of the lemma. Furthermore, as 
$ \Pi  \cap \overline{K_{\bt,Z(\bt)}} = \emptyset$ by definition,
$|\Pi \cap K_{\bt,\bz}^{(m)} \setminus \Pi_2|>0$ and $Z(\bt) \in A_m(\bz)$ 
imply that $\Pi \cap K_{\bt,\bz}^{(m)}$ contains points which can be removed
without affecting $Z(\bt) \in A_m(\bz)$. Thus,
\begin{align*}
& \lim_{m\to\infty} \Prob\left(|\Pi \cap K_{\bt,\bz}^{(m)} \setminus \Pi_2|>0 \ \big| \ Z(\bt) \in A_m(\bz)\right) \displaybreak[0]\\
\leq{} & \lim_{m\to\infty} \textstyle \sum_{\bn \in N_2} \Prob\left(E^{(m)}_{\bn}(\bz) \ \big| \ Z(\bt) \in A_m(\bz)\right) = 0,
\end{align*}
which verifies the second assertion of this lemma. \qed
\end{proof}
\medskip

From Lemma \ref{only-nec-points}, we immediately get the following assertion.

\begin{corollary} \label{as-onepoint}
 For any fixed $t_1,\ldots,t_n \in \R^d$ we have
 $$\Prob(|\Pi \cap K_{t_i,Z(t_i)}| \geq 2 \textit{ for some } i \in \{1,\ldots,n\}) = 0.$$
\end{corollary}
\begin{proof}
 It suffices to show $\Prob(|\Pi \cap K_{t_i,Z(t_i)}| \geq 2) = 0$ for all 
 $i \in \{1,\ldots,n\}$. Now, let $i \in \{1,\ldots,n\}$ be fixed. Then,
 conditioning on $Z(t_i)$ only, we get
\begin{align*}
\Prob\left(|\Pi \cap K_{t_i,Z(t_i)} | \geq 2 \mid Z(t_i) = z\right)
 ={} \textstyle \Prob\left(\bigcup_{\bn \in N_2} E_{\bn}(Z(t_i)) \ \big| \ Z(t_i) = z\right) = 0
\end{align*}
for almost every $z >0$ by Lemma \ref{only-nec-points}. \qed
\end{proof}

Corollary \ref{as-onepoint} ensures that almost surely one of the scenarios
$E_{\bn}(\bz)$ with $\sum_{A: \, A \ni i} n_A = 1$ for all 
$i \in \{1,\ldots,n\}$ occurs. This is also stated in a more general setting in
\cite{dombry-2011}, Prop.\ 2.2.
\medskip

\begin{figure}\begin{center}
\includegraphics[height=3.9cm,width=5.2cm]{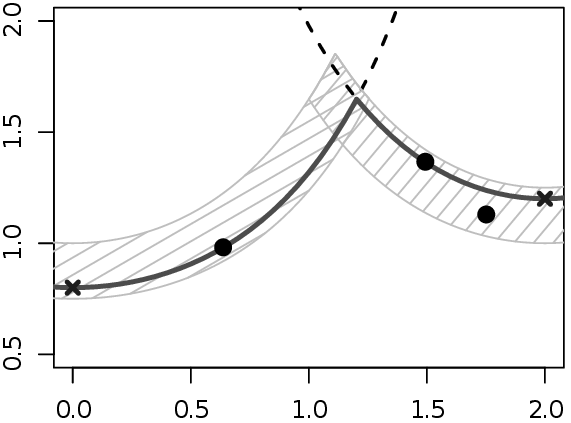} \hfill
\includegraphics[height=3.9cm,width=5.2cm]{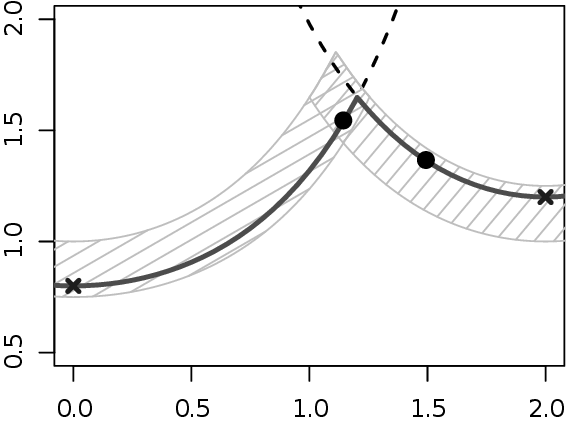}
\caption{The event $\Delta_m$ in the proof of Thm.\ \ref{generalconvergence}.
          Black crosses: data $(t_i,Z(t_i))$ with $t_1 = 0$ and $t_2=2$, 
          black line: $K_{\bt,Z(\bt)}$, grey hatched area: $K^{(m)}_{\bt,Z(\bt)}$,
          black dots: $\Pi \cap K_{\bt,Z(\bt)}^{(m)}$. \newline
          Left: $|\Pi \cap K_{\bt,Z(\bt)}^{(m)} \setminus \Pi_2| > 0$.
          Right: $|\Pi \cap I_{\{1\}}(Z(\bt))| = 1  > 0 = |\Pi \cap I_{\{1\}}^{(m)}(Z(\bt))|$.}
          \label{fig:wrong_points}
\end{center}
\vspace{-0.4cm}
\end{figure}

\begin{theorem} \label{generalconvergence}
For $\Prob_{Z(\bt)}$-a.e.~$\bz>\mathbf{0}$, we have
 \begin{align*}
 1.{} \ &\Prob(E_{\bn}(Z(\bt)) \mid Z(\bt) = \bz) {}={} \lim_{m \to \infty} \Prob(E^{(m)}_{\bn}(\bz)
      \mid Z(\bt) \in A_m(\bz)) \textrm{ for all } \bn \in N_1,\\
 2.{} \ & \Prob\big(|\Pi_2 \cap B_j| = r_j, \, 1 \leq j \leq k, \ |\Pi_3 \cap B_j| = r_j, \ k < j \leq l 
        \ \big| \ Z(\bt) = \bz \big) \displaybreak[0] \\
      & {}={} \lim_{m\to \infty} \Prob\big(|\Pi \cap K_{\bt,\bz}^{(m)} \cap B_j| = r_j, \ 1 \leq j \leq k, \\
      & \hspace{1.8cm} |\Pi \cap B_j \setminus (K_{\bt,\bz}^{(m)} \cup \overline{K_{\bt,\bz}}^{(m)})| = r_j,
          \ k < j \leq l \ \big| \ Z(\bt) \in A_m(\bz)\big)\\
     & \quad \textrm{for any } B_j \subset S, \ r_j \in \N_0, \ j \in \{1,\ldots,l\}.
 \end{align*}
\end{theorem}
\begin{proof}
For the proof of the first part, note that, by L\'evy's ``Upward'' Theorem
 \begin{align*}
 \Prob(E_{\bn}(Z(\bt)) \mid Z(\bt) = \bz) {}={} \lim\nolimits_{m \to \infty} \Prob(E_{\bn}(Z(\bt))
 \mid Z(\bt) \in A_m(\bz))
 \end{align*}
for $\Prob_{Z(\bt)}$-a.e.~$\bz>\mathbf{0}$ and all $\bn \in N_1$. 
Thus, as $\Prob(E^{(m)}_{\bn}(Z(\bt)) \mid Z(\bt) \in A_m(\bz))$ equals 
$\Prob(E^{(m)}_{\bn}(\bz) \mid Z(\bt) \in A_m(\bz))$ by definition, it remains
to verify
\begin{align}
 \lim_{m \to \infty} \big|\Prob(E^{(m)}_{\bn}(Z(\bt)) \,|\, Z(\bt) \in A_m(\bz))
 - \Prob(E_{\bn}(Z(\bt)) \,|\, Z(\bt) \in A_m(\bz)) \big| = 0. \label{eq:pfgenconv}
\end{align}
To this end, we consider the symmetric difference $\Delta_m$ of 
$E^{(m)}_{\bn}(Z(\bt))$ and $E_{\bn}(Z(\bt))$. Note that any element of 
$\Delta_m$ satisfies $|\Pi \cap K_{\bt,Z(\bt)}^{(m)} \setminus \Pi_2| > 0$ or 
$|\Pi \cap I_A(Z(\bt))| > |\Pi \cap I_A^{(m)}(Z(\bt))|$ for some $A \in \AA$
(cf.\ Figure \ref{fig:wrong_points}). The second kind of event happens if there
is a point of $\Pi$ in 
$I_A(Z(\bt)) \cap (\bigcup_{j \notin A} K_{t_j,Z(t_j)}^{(m)})$. As this set
vanishes for any $Z(\bt) > \mathbf{0}$ as $m \to \infty$, we get that
$$ U_A^{(m)} = \{\omega \in \Omega: \ |\Pi \cap I_A(Z(\bt))| > |\Pi \cap I_A^{(m)}(Z(\bt))|\} \searrow \emptyset, \quad m \to \infty,$$
for any $A \in\AA$, and therefore 
$\textstyle \lim_{m \to \infty} \Prob(U_A^{(m)}) = 0$.
This yields
\begin{align*}
  \int_{(0,\infty)^n} & \lim_{m\to\infty}  \Prob\left(U_A^{(m)} \, \big| \, Z(\bt) \in A_m(\bz)\right)
  \Prob(Z(\bt) \in \rd \bz) = 0
\end{align*}
by dominated convergence and the fact that
\begin{align*}
& \int_{(0,\infty)^n} \Prob\left(U_A^{(m)} \ \big| \ Z(\bt) \in A_m(\bz)\right) \Prob(Z(\bt) \in \rd \bz) \displaybreak[0]\\
={} & \textstyle \sum_{\bz \in (2^{-m}\N)^n} \Prob\left(U_A^{(m)} \ \big| \ Z(\bt) \in A_m(\bz)\right) \cdot \Prob(Z(\bt) \in A_m(\bz))
{}={}  \Prob\big(U_A^{(m)}\big).
\end{align*}

Therefore, we have
\begin{equation} \label{eq:wrongpoints}
 \lim_{m \to \infty} \Prob\left(|\Pi \cap I_A(Z(\bt))| > |\Pi \cap I_A^{(m)}(\bz)| \ \big| \ Z(\bt) \in A_m(\bz)\right) = 0
\end{equation}
for any $A \in \AA$ and $\Prob_{Z(\bt)}$-a.e.~$\bz > \mathbf{0}$.
All in all, we end up with
\begin{align*}
 & \lim_{m \to \infty} \Big|\Prob\left(E^{(m)}_{\bn}(Z(\bt)) \ \big| \ Z(\bt) \in A_m(\bz)\right)
 - \Prob\left(E_{\bn}(Z(\bt)) \ \big| \ Z(\bt) \in A_m(\bz)\right)\Big|\\
\leq{} & \quad \lim_{m \to \infty}  \sum_{A \in \AA} \Prob\left( U_A^{(m)} \, \big| \, Z(\bt) \in A_m(\bz)\right)\\
    & + \lim_{m \to \infty} \Prob\left(|\Pi \cap K^{(m)}_{\bt,\bz} \setminus \Pi_2| > 0 \, \big| \, Z(\bt) \in A_m(\bz)\right) {}={} 0
\end{align*}
by \eqref{eq:wrongpoints} and by the second part of Lemma 
\ref{only-nec-points}. Thus, we get
\eqref{eq:pfgenconv}. 

For the proof of the second assertion, let $B_1, \ldots, B_k, 
B_{k+1}, \ldots, B_l \subset S$ be Borel sets. Then, each of the events
\begin{align*}
& \{|\Pi_2 \cap B_j| \neq |\Pi \cap K_{\bt,Z(\bt)}^{(m)} \cap B_j| \textrm{ for any } j=1, \ldots,k\}
  \displaybreak[0]\\
\textrm{and } &
\{|\Pi_3 \cap B_j| \neq |\Pi \cap B_j \setminus 
(K_{\bt,Z(\bt)}^{(m)} \cup \overline{K_{\bt,Z(\bt)}}^{(m)})| \textrm{ for any } j=k+1, \ldots, l\}
\end{align*}
implies that $|\Pi \cap K_{\bt,Z(\bt)}^{(m)} \setminus \Pi_2| > 0$.
Hence, by the second part of Lemma \ref{only-nec-points},
\begin{align*}
 \lim_{m \to \infty} & \Big|\Prob\big(|\Pi_2 \cap B_{j}|=r_{j}, \, 1 \leq j \leq k, \,
 |\Pi_3 \cap B_j|=r_j, \, k < j \leq l \, \big| \, Z(\bt) \in A_m(\bz)\big)\\
 & \hspace{0.7cm} - \Prob\big(|\Pi \cap K^{(m)}_{\bt,\bz} \cap B_j|=r_j, \ 1 \leq j \leq k,\\
 & \hspace{1.5cm}
  |\Pi \cap B_j \setminus (K^{(m)}_{\bt,\bz} \cup \overline{K_{\bt,\bz}}^{(m)})| = r_j,
  \ k < j \leq l \,  \big| \, Z(\bt) \in A_m(\bz)\big)\Big|\\
\leq{} & \lim_{m \to \infty} \Prob\left(|\Pi \cap K^{(m)}_{\bt,\bz} \setminus \Pi_2| > 0 \ \big| \ Z(\bt) \in A_m(\bz)\right) = 0.
\end{align*}
for any $r_1, \ldots, r_k$, $r_{k+1}, \ldots, r_l \in \N_0$. \qed
\end{proof}
\medskip

We note that there exists a more general version of Theorem 
\ref{generalconvergence} which we will need for simulation. 
Let $B_1, \ldots, B_k \in \BB^d \times ((0,\infty) \cap \BB) \times \GG$ be
pairwise disjoint with  $\bigcup_{j=1}^k B_j =S$.
We introduce \emph{generalized ``blurred'' scenarios}
\begin{align*}
 E^{(m)}_{\bn,\mathbf{B}}(\bz) {}={} & \{|\Pi \cap I_A^{(m)}(\bz) \cap B_j| 
 = n_A^{(j)}, \, A \in \AA, \, 1 \leq j \leq k, \, |\Pi \cap \overline{K_{\bt,\bz}}^{(m)}| = 0\}
\end{align*}
where $\bn = (n_A^{(j)})_{A \in \AA,\ j=1,\ldots,k} \in \N_0^{(2^n-1)k}$ 
satisfies $\sum_{j=1}^k \sum_{A: \, A \ni i} n_A^{(j)} \geq 1$ for all 
$i \in \{1,\ldots,n\}$. Analogously, \emph{generalized scenarios} 
$E_{\bn,\mathbf{B}}(\bz)$ are defined. Then, in the same way as Theorem
\ref{generalconvergence} the following theorem can be shown.

\begin{theorem} \label{verygeneralconvergence}
For $\Prob_{Z(\bt)}$-a.e.~$\bz>\mathbf{0}$, we have
 \begin{equation*}
 \textstyle \Prob(E_{\bn,\mathbf{B}}(Z(\bt)) \mid Z(\bt) = \bz) 
 ={} \lim_{m \to \infty} \Prob(E^{(m)}_{\bn,\mathbf{B}}(\bz) \mid Z(\bt) \in A_m(\bz))
 \end{equation*}
for any scenario $E_{\bn,\mathbf{B}}(\bz)$ with 
$\sum_{j=1}^k \sum_{A: \, A \ni i} n_A^{(j)} \geq 1$ for all 
$i \in \{1,\ldots,n\}$.
\end{theorem}
\medskip

By a straightforward application of Theorems \ref{generalconvergence} and 
\ref{verygeneralconvergence} the conditional independence of $\Pi_1$,
$\Pi_2$ and $\Pi_3$ can be shown and the conditional distribution of $\Pi_3$
can be calculated. We skip the proof and refer to  \cite{dombry-2011}, 
Thm.\ 3.1, who obtained these results by a different approach, using Palm
theory. However, as Palm theory is not helpful for the calculation of the
distribution of $\Pi_2$, we use the approximation method for all the results.

\begin{corollary} \label{condindependence}
\begin{enumerate}
\item  With probability one the point processes $\Pi_1$, $\Pi_2$ and $\Pi_3$
  conditional on $Z(\bt)$  are stochastically independent.
\item The process $\Pi_3 \mid Z(\bt)=\bz$ has the same distribution as 
  $\Pi \setminus (K_{\bt,\bz} \cup \overline{K_{\bt,\bz}})$ for 
  $\Prob_{Z(\bt)}$-a.e.~$\bz>\mathbf{0}$.
\end{enumerate}
\end{corollary}

By the second part of Corollary \ref{condindependence} the process 
$\Pi_3 \mid Z(\bt)=\bz$ can be easily simulated by unconditionally simulating
$\Pi$ and restricting it to 
$\R^d \times (0,\infty) \times G \setminus (K_{\bt,\bz} \cup \overline{K_{\bt,\bz}})$.
\medskip

\begin{remark} \label{rem:stationarity}
 Note that, by the definition of $\Lambda$ in \eqref{eq:intensity}, the process
 $Z$ in \eqref{eq:procdef} is stationary. Replacing the Lebesgue measure $\mu$
 by an arbitrary absolutely continuous measure yields non-stationary models, as
 well. All the results presented so far still hold true for these processes. 
 The formulae that will occur in Section \ref{sec-finite} have to be modified
 using the corresponding Lebesgue density in case of non-stationarity. 
 The assumption that the distribution of the shape function does not depend on
 the other components of $\Pi$, however, is crucial. That is, the law of the 
 shape function is independent of the shifting and scaling.
\end{remark}
\medskip

The remainder of the paper will address the problem of simulating 
$\Pi_2 \mid Z(\bt)=z$. We propose a procedure  consisting of two steps. First,
we draw a scenario $E_{\bn}(Z(\bt))$ conditional on $Z(\bt)=\bz$. Then, the
points of $\Pi_2$ corresponding to this scenario are simulated.
\medskip

\section{Results in the case of a finite number of shape functions on the real 
         axis} \label{sec-finite}

As shown in Section \ref{sec-limits}, the calculation of
$\Prob(E_{\bn}(Z(\bt)) \mid Z(\bt) = \bz)$ requires knowledge about
the exact asymptotic behaviour of $\Lambda(I_A^{(m)}(\bz))$. 
In particular, the behaviour of the intersection of two curves 
$K_{t_i,z_i+ \delta_i} \cap K_{t_j,z_j + \delta_j}$ for small $|\delta_i|$
and $|\delta_j|$ needs to be analysed. Explicit calculations turn out to be
quite laborious. Therefore, we restrict ourselves to the case $d=1$. 
We calculate the asymptotics of $\Lambda(I_A^{(m)}(\bz))$ for
$|A|=1$ (Proposition \ref{Kintens}), $|A|=2$ (Proposition \ref{intersectintens})
and $|A|\geq 3$ (Proposition \ref{3points}), see Figure \ref{fig:blurred_curves}.
In the case $|A| \geq 3$, the rate of convergence of $\Lambda(I_A^{(m)})$ 
cannot be determined exactly.  Nevertheless, the conditional probability of any
scenario can be calculated (Theorem \ref{wellcalculatable}). All the proofs of
this section are postponed to Section \ref{sec-calculate}.
\medskip

\begin{figure}\centering
\includegraphics[height=3.6cm,width=3.6cm]{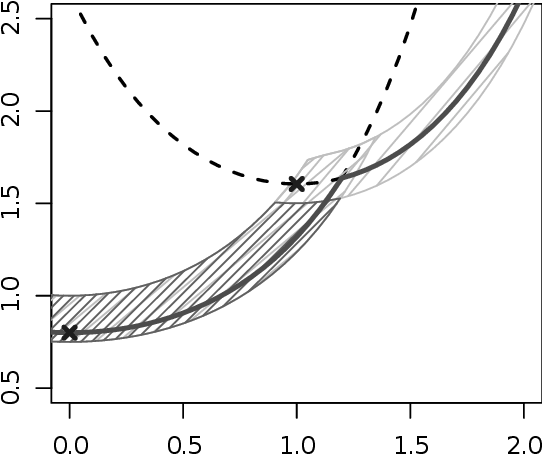} \hfill
\includegraphics[height=3.6cm,width=3.6cm]{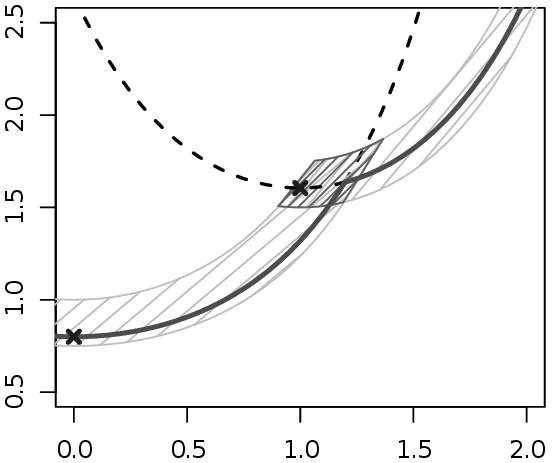} \hfill
\includegraphics[height=3.6cm,width=3.6cm]{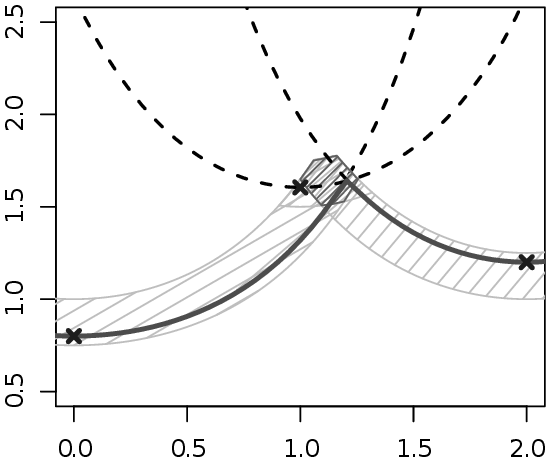}
\caption{Blurred intersection sets for $|A|=1$, $|A|=2$ and $|A|=3$.
     Black crosses: data $(t_i,Z(t_i))$, dashed black lines: $K_ {t_i,Z(t_i)}$,
     black line: $K_{\bt,Z(\bt)}$, grey hatched area: $K^{(m)}_{\bt,Z(\bt)}$,
     black hatched area:  $I_A^{(m)}(Z(\bt))$ with $|A|=1$ (left), $|A|=2$ 
     (middle) and $|A|=3$ (right).} \label{fig:blurred_curves}
\vspace{-0.2cm}
\end{figure}

First we assume that $G$ is a finite space of functions $f: \R \to [0,\infty)$
such that the intersections
\begin{align} \label{eq:intersection}
& M_{c,t_0}= \{ t \in \R: \ f(t) = c f(t_0+t), \ f(t) > 0\}
\end{align}
are finite for all $c>0$, $t_0 \in \R$, $f \in G$.
This implies that each set $I_A(\bz)$, $A \in \AA$, $|A| \geq 2$, 
$\bz > \mathbf{0}$, is finite. Without loss of generality, we assume that
$\Prob_F(\{f\}) > 0$ for all $f \in G$.

The following Propositions \ref{intersectintens} and \ref{3points}, deal with
the asymptotic behaviour of $\Lambda(I^{(m)}_A(\bz))$, 
$|A| > 1$, as $m \to \infty$. For simplicity, we assume that $I_A(\bz)$
consists of a single point. In general, $I_A(\bz)$ is finite
and, for $m$ large enough, $I_A^{(m)}(\bz)$ consists of $|I_A(\bz)|$ 
disjoint components.
Thus, $\Lambda(I^{(m)}_A(\bz))$ is the sum of the measures belonging to
each component (cf.\ Remark \ref{more-2points} and Proposition
\ref{more-3points}).

\begin{proposition} \label{intersectintens}
Let $t_1, t_2 \in \R$, $z_1, z_2 > 0$ such that 
$ I_{\{1,2\}}(\bz) = \{(t_0,y_0,f)\}.$ Furthermore, let $f$ be continuously
differentiable in a neighbourhood of $t_1-t_0$ and $t_2-t_0$ with
\begin{equation} \label{eq:invertibilitycrit}
 z_1 f'(t_2-t_0) \neq z_2 f'(t_1-t_0).
\end{equation}
Then, we have 
$$\Lambda(I_{\{1,2\}}^{(m)}(\bz)) 
= \frac{2^{-2m}}{y_0^2 |z_1 f'(t_2-t_0) - z_2 f'(t_1-t_0)|} \Prob_F(\{f\}) + o(2^{-2m}).$$
\end{proposition}
\medskip

\begin{remark} \label{more-2points}
\begin{enumerate}
\item[(i)] Using $\frac{z_1}{f(t_1-t_0)} = \frac{z_2}{f(t_2-t_0)} = y_0$ we get
that equality in \eqref{eq:invertibilitycrit} holds if and only if
$$ \frac{\partial}{\partial t} \frac{z_1}{f(t_1-t)} \bigg|_{t=t_0}
= \frac{\partial}{\partial t} \frac{z_2}{f(t_2-t)} \bigg|_{t=t_0},$$
i.e.\ if and only if the two sets of admissible points,
$K_{t_1,z_1}$ and $K_{t_2,z_2}$, are tangents to each other in $(t_0,y_0,f)$
which is an event of probability zero by Assumption \eqref{eq:intersection}.
Therefore, \eqref{eq:invertibilitycrit} is satisfied $\Prob_{Z(\bt)}$-a.s.
\item[(ii)] If 
$ I_{\{1,2\}}(z_1,z_2) = \{(s_1,u_1,f_1), \ldots, (s_k,u_k,f_k)\},$
we obtain $$ \textstyle \Lambda(I^{(m)}_{\{1,2\}}(\bz)) = 2^{-2m} \cdot 
\sum_{j=1}^k \frac {\Prob_F(\{f_j\})} {u_j^2 \cdot |z_1 f_j'(t_2-s_j) - z_2 f_j'(t_1 - s_j)|} + o(2^{-2m}).$$
\end{enumerate}
\end{remark}

\begin{proposition} \label{3points}
Let $t_1, \ldots, t_l \in \R$, $z_1, \ldots, z_l >0$, $l \geq 3$ such that 
$I_{\{1,\ldots,l\}}(\bz) = \{(t_0,y_0,f)\}$
where $f$ is continuously differentiable in a neighbourhood
of $t_1-t_0,\ldots,t_l-t_0$ such that \eqref{eq:invertibilitycrit} holds for
$(t_1,z_1)$ and $(t_2,z_2)$.
Then, we have
$$\Lambda(I_{\{1,\ldots,l\}}^{(m)}(\bz)) \leq{}
\frac{2^{-2m}}{y_0^2 |z_1 f'(t_2-t_0)-z_2 f'(t_1-t_0)|} \Prob_F(\{f\})
+ o(2^{-2m}).$$
For any $C > 0$, $\varepsilon > 0$, there exists $m_{C,\varepsilon} \in \N$ 
such that
$$\Lambda(I_{\{1,\ldots,l\}}^{(m)}(\bz)) \geq{} C 2^{-2m(1+\varepsilon)}, \qquad m \geq m_{C,\varepsilon}.$$
\end{proposition}

Whilst Proposition \ref{3points} does not provide exact asymptotics,
we get exact results in the following situation.

\begin{proposition} \label{more-3points}
Let $I_{\{1,\ldots,l\}}(\bz) = \{(s_1, u_1, f_1), \ldots, (s_k, u_k, f_k)\}$,
$l \geq 3$, 
where $f_1, \ldots, f_k: \R \to [0,\infty)$ are continuously differentiable
in a neighbourhood of $t_i - s_j$, $i=1,\ldots,l$, $j=1,\ldots,k$ such
that \eqref{eq:invertibilitycrit} holds for $(t_1,z_1)$, $(t_2,z_2)$ and each
$(s_j, y_j, f_j)$, $j=1,\ldots,k$. Then, we have
\begin{align}
 & \Prob((t_0,y_0,f) \in \Pi \mid |\Pi \cap I_{\{1,\ldots,l\}}(Z(\bt))| = 1, \, Z(\bt)=\bz) \nonumber \\
= & \frac{y_0^{-2} \Prob_F(\{f\})}{|z_1 f'(t_2-t_0) - z_2 f'(t_1-t_0)|} 
\bigg(\sum_{j=1}^{k} \frac{u_j^{-2} \Prob_F(\{f_{j}\})}{|z_1 f'_{j}(t_2-s_j)-z_2 f'_{j}(t_1-s_j)|}\bigg)^{-1}
\label{eq:more-3points}
\end{align}
for $(t_0,y_0,f) \in I_{\{1,\ldots,l\}}(\bz)$. Furthermore, for 
$\Prob_{Z(\bt)}$-a.e.\ $\bz > \mathbf{0}$, the right-hand side of 
\eqref{eq:more-3points} does not depend on the choice of the labelling.
\end{proposition}

\begin{proposition} \label{Kintens}
Let $f \in G$, $\bt \in \R^n$, $\bz > \mathbf{0}$ such that $f$ is continuously
differentiable in a neighbourhood of $t_i-t_0$ for all $i \in \{1, \ldots, n\}$
and all $(t_0,y_0,f) \in K_{\bt,\bz} \cap \bigcup_{\substack{l=1\\l\neq i}}^n (K_{t_i,z_i} \cap K_{t_l,z_l})$,
i.e.\ all $(t_0,y_0,f) \in K_{t_i,z_i}$ that generate at least two observations.
Furthermore, we denote the projection of the set
$I_{\{i\}}(\bz) \cap (\R \times (0,\infty) \times \{f\})$
onto its first component in $\R$ by
$$D_{i}^{(f)} = \{ t \in \R: \ (t,y,f) \in I_{\{i\}}(\bz) \textrm{ for some } y>0\}.$$
Then, with $S_f =  \R \times (0,\infty) \times\{f\}$, we have 
\begin{equation} \label{eq:singlecurve}
 \Lambda\left(I_{\{i\}}^{(m)}(\bz) \cap S_f\right) 
={} 2^{-m} \cdot \Prob_F(\{f\}) \cdot z_i^{-2} \int_{D_i^{(f)}} f(t_i-t) \sd t + o(2^{-m}).
\end{equation}
\end{proposition}
\medskip

We can use Theorem \ref{generalconvergence} in order to compute the conditional
probabilities
 \begin{align}
   \Prob(E_{\bn}(Z(\bt)) \mid Z(\bt) = \bz) 
={} & \lim_{m \to \infty} \frac{\Prob(E^{(m)}_{\bn}(\bz))}
      {\sum_{\tilde \bn \in N_0} \Prob(E^{(m)}_{\tilde \bn}(\bz))}
      \label{eq:generalconvergence}
 \end{align}
where 
$N_0 = \{ (n_A)_{A \in \AA}: \ \sum_{A: \, A \ni i} n_A =1, \ i=1,\ldots,n\}$.

As the sets $I_A^{(m)}(\bz)$, $A \in \AA$, are pairwise disjoint and 
$\Lambda(I_A^{(m)}(\bz))$ tends to zero for $m \to \infty$ by Lemma
\ref{lambdaconv}, we get
\begin{align}
\Prob(E^{(m)}_{\bn}(\bz)) 
\sim{} & \textstyle \exp(-\Lambda(\overline{K_{\bt,\bz}})) \prod_{A: \, n_A=1} \Lambda(I_A(\bz)) .
\end{align}

Considering \eqref{eq:generalconvergence}, we can restrict ourselves to those
scenarios with the slowest rate of convergence to zero. Propositions 
\ref{intersectintens}, \ref{3points} and \ref{Kintens} yield that scenarios
involving intersections of at least three sets are always of a dominating order.
Therefore, all unknown terms from Proposition \ref{3points} appear as factors in
the numerator and denominator in \eqref{eq:generalconvergence} and hence are
cancelled out.
Using the formulae above, the limits of the conditional probabilities can always
be calculated explicitly except for those cases where two scenarios exist, both
involving different terms which cannot be determined exactly (cf.\ Proposition
\ref{3points}). This may happen only if two sets $A_1=\{i_1,\ldots,i_r\}$ and
$A_2=\{ j_1,\ldots,j_s\}$, $r,s \geq 3$,  $A_1 \cap A_2 \neq \emptyset$, exist
such that
\begin{align}
J_{A_1}(Z(\bt)) \neq \emptyset \ , \ J_{A_2} (Z(\bt)) \neq \emptyset
\textrm{ and } J_{A_1 \cup A_2}(Z(\bt)) = \emptyset \label{eq:problems},
\end{align}
where $ J_A(Z(\bt)) = \bigcup_{B \supset A} I_B(Z(\bt)) = K_{\bt,Z(\bt)} 
        \cap \bigcap_{i \in A} K_{t_i,Z(t_i)}$, $A \in \AA$.

In all other cases the terms as in Proposition \ref{3points} are cancelled out.
Note that we work with sets of the type $J_A(Z(\bt))$ in order to avoid 
case-by-case analysis for all the sets $I_B(Z(\bt))$ with $B \supset A$.

\begin{lemma} \label{neglectprobs}
Let $G$ consist of functions which are continuously differentiable a.e.
Then, for any fixed set $\{t_1,\ldots,t_n\} \subset \R$ we have
$$\Prob(Z(\bt) \textrm{ satisfies } \eqref{eq:problems}) = 0.$$
\end{lemma}
\medskip

Thus, from the considerations above we directly derive the following result.

\begin{theorem} \label{wellcalculatable}
 Let $G$ be a finite set of functions which are continuously differentiable 
 a.e. Then, for $\Prob_{Z(\bt)}$-a.e.~$\bz>\mathbf{0}$,
\begin{align*}
 \Prob(E_{\bn}(Z(\bt)) \mid Z(\bt) = \bz) 
{}={} & \lim_{m \to \infty} \Prob(E^{(m)}_{\bn}(\bz) \mid Z(\bt) \in A_m(\bz))
\end{align*}
can be calculated explicitly by the results of Propositions 
\ref{intersectintens}, \ref{3points}, \ref{more-3points} and \ref{Kintens}.
\end{theorem}

\begin{remark}
We may also consider the case that $G$ is countable. However, to transfer the
results of the finite case, we have to ensure uniform convergence of the 
blurred intersection sets which is needed to compute the term
$\sum_{ \bn \in N_0} \Prob(E^{(m)}_{\bn}(\bz))$ in the denominator of
Equation \eqref{eq:generalconvergence}. To this end, we have to impose some
additional conditions. For example, we could assume that for almost every
$\bz > \mathbf{0}$ there is only a finite number of shape functions
involved in the intersection sets $I_A(\bz)$, $|A| \geq 2$.
\end{remark}
\medskip

We are still left with simulating $\Pi_2 \mid Z(\bt)=\bz$ given the occurrence
of a scenario $E_{\bn}(\bz)$ with $\bn \in N_0$, that is, we are interested in
$$ \textstyle \Prob\left(\bigcap_{A: \, n_A=1} \{|\Pi_2 \cap (C_A \times (0,\infty) \times \{f\})| =1 \} \ \big|
 \ E_{\bn}(\bz)\right)$$
for $C_A \subset \R$, $f \in G$ with $(C_A \times (0,\infty) \times \{f\}) \cap I_A(\bz) \neq \emptyset$.
Using Theorem \ref{verygeneralconvergence} with sets $B_A = C_A \times (0,\infty) \times \{f\}$, 
$A \in \AA$, we get that
\begin{align*}
 \Prob\bigg(\bigcap_{A: n_A=1} \{|\Pi_2 \cap I_A(\bz) \cap B_A| =1\} \, \big| \, E_{\bn}(\bz)\bigg)
= \lim_{m\to\infty} \prod_{A: n_A =1} \frac{\Lambda(I_A^{(m)}(\bz) \cap B_A)}{\Lambda(I_A^{(m)}(\bz))}.
\end{align*}
Thus, each point of $\Pi_2$ can be simulated independently.
If $\Pi_2 \cap I_A(Z(\bt))$ contains exactly one point, define $T_A \in \R$ and
$F_A \in G$ such that
\begin{equation}
 \Pi_2 \cap I_A(Z(\bt)) = \{(T_A, F_A(t_{i^*}-T_A)/Z(t_{i^*}),F_A)\}
\end{equation}
for any arbitrary $i^* \in A$.
Note that the distribution of $(T_A, F_A)$ depends on the cardinal number of
$A$. If $A=\{i\}$ for some $i \in \{1,\ldots,n\}$, we have
$$\Prob(T_A \in B, F_A=f) = \frac{\Prob_F(\{f\}) \int_{D_i^{(f)} \cap B} f(t_i-t) \sd t}
{\sum_{g \in G} \Prob_F(\{g\}) \int_{D_i^{(g)}} g(t_i-t) \sd t}, \quad B \in \BB.$$
For $|A| \geq 2$, let $I_A(\bz) = \{(s_1,u_1,f_1), \ldots, (s_k,u_k,f_k)\}$.
Then, we get
$$\Prob(T_A = s_j, F_A=f_j) =
 \frac{\frac{u_j^{-2} \Prob_F(\{f_j\})}{ |z_1 f'_j(t_2-s_j) - z_2 f'_j(t_1-s_j)|}}
{\sum_{l=1}^{k} \frac{u_l^{-2}\Prob_F(\{f_l\})}{|z_1 f'_l(t_2-s_l)-z_2 f'_l(t_1-s_l)|}},
\quad 1 \leq j \leq k.$$
\medskip

Thus, we end up with the following procedure for calculating the conditional
distribution of $Z(\cdot)$ given $Z(\bt)$ with $t_1,\ldots,t_n \in \R$,
$\bz > \mathbf{0}$.
\begin{enumerate}
 \item Compute the conditional probabilities \eqref{eq:generalconvergence} for
       all the scenarios $E_{\bn}(\bz)$ and generate a random scenario 
       following this distribution.
 \item For a given scenario $E_{\bn}(\bz)$ simulate 
       $\Pi_2 = \{ (T_A, \frac{z_{i^*}}{F_A(t_{i^*}-T_A)},F_A): \ n_A = 1\}$
       for an arbitrary $i^* \in A$. Here, the law of $(T_A,F_A)$
       is given above.
 \item Independently, sample from 
       $\Pi_3 = \Pi \cap (\R \times (0,\infty)\times G) \setminus (K_{\bt,\bz}
        \cup \overline{K_{\bt,\bz}})$.\\
 Then, $Z(\cdot) = \max_{(s,u,f) \in \Pi_2 \cup \Pi_3} u f(\cdot -s)$.
\end{enumerate}

In the next section, we will demonstrate the performance of this exact approach
by comparing it to other algorithms in the simple case of a deterministic,
continuously differentiable shape function.
\medskip

\section{Comparison with the algorithms for the max-linear model and 
         for Gaussian processes with transformed marginals}
\label{sec-compare}

Recently, \cite{wang-stoev-2010} proposed an algorithm for exact and efficient
conditional sampling for max-linear models
$$ Z(t_i) = \max_{j=1,\ldots,p} a_{ij} Y_j, \quad i=0,\ldots,n,$$
where $Y_j, \ j=1,\ldots,p,$ are independent standard Fr\'echet random variables.
With the representation \eqref{eq:extremal-integral} of $Z$ as an extremal
integral we see that $Z$ can be approximated arbitrarily well by a max-linear
model, e.g.\ by
$$Z_{M,h}(t) = h \max_{\substack{l= -M, \ldots, M-1\\j=1,\ldots,k}} \Prob_F(\{f_j\}) 
\cdot f_j\Big(t-\Big(l+\frac 1 2\Big)h\Big) \cdot Y_{j,l}, \quad M \in \N, \ h > 0,$$
where $G=\{f_1,\ldots,f_k\}$.
Then, we have $Z_{M,h}(t) \stackrel{P}{\longrightarrow} Z(t)$ for any 
$t \in \R$ as $M \to \infty$, $h \to 0$.
\medskip

We also consider another approach based on the assumption of a multi-Gaussian 
model \citep[cf.][p.\ 381]{chiles-delfiner-1999}. The data are transformed
such that the marginal distribution is Gaussian. As the marginals of $Z$ are
standard Fr{\'e}chet, the corresponding transformation is given by 
$$\Psi: (0,\infty) \to \R, \ x \mapsto \Phi^{-1}(\Phi_1(x)),$$
where $\Phi$ is the standard normal distribution function and 
$\Phi_1 = \exp(-1/x)$ is the standard Fr{\'e}chet distribution function.
The transformed random field $Y=\Psi(Z)$ is stationary and has second-order 
moments. As the computation of covariance function $C$ of $Y$ for general 
shape functions $f_1, \ldots, f_k$ is complex, we estimate $C$ using maximum
likelihood techniques, for instance, from a convenient parametric class such as
the Whittle-Mat\'ern class, i.e.
\begin{equation} \label{eq:whittle}
  C_{\nu, c} (h) = \frac {(c||h||)^\nu}{2^{\nu-1}\Gamma(\nu)} 
   \mathcal{K}_\nu(c||h||), \quad \nu, \ c > 0,
\end{equation}
assuming that $Y$ is a Gaussian random field. Under this assumption, the 
conditional distribution can be sampled easily 
\citep[see][for instance]{lantuejoul-2002}.
Afterwards, the sample has to be retransformed via
$$\Psi^{-1}: \R \to (0,\infty),\ y \mapsto \Phi_1^{-1}(\Phi(y)).$$
Note that, in general, this procedure is not exact as $Y$ is not a Gaussian
random field, but only marginally Gaussian.

To compare these different methods, we need a measure for the goodness-of-fit
of a distribution. Here, we use the \emph{continuous ranked probability score}
(CRPS) which is defined as 
$$ CRPS(F_1,x) = - \int_{-\infty}^{\infty}(F_1(y)-\mathbf{1}_{\{y\geq x\}})^2\sd y,$$
where $F_1$ is a cumulative distribution function and $x \in \R$
\citep{gneiting-raftery-2007}.
Note that $CRPS(F_1,F_2) := \int CRPS(F_1,x) F_2({\rm d}x)$ is a strictly
proper scoring rule, i.e.
$$CRPS(F_2,F_2) \geq CRPS(F_1,F_2)$$ for all cumulative distribution functions 
$F_1$, $F_2$. If both $F_1$ and $F_2$ belong to measures with finite first
moment, equality holds if and only if $F_1=F_2$. Assuming that $F_1$ has a
finite first moment, the CRPS can be calculated via
\begin{equation} \label{eq:CRPS}
CRPS(F_1,x) = \frac 1 2 \E|X-X'| - \E|X-x|,
\end{equation}
which shows that $CRPS(F_1,F_1) = - \frac 1 2 \E|X-X'| \leq 0$. Here, $X$, 
$X' \sim F_1$ are independent random variables.
\medskip

In order to compare different algorithms that simulate from the conditional
distribution $\log(Z(t_0)) \mid Z(\bt)$, we consider $K$ samples
$Z_1,\ldots,Z_K$ of the random field $Z$. For each method, we get an empirical
distribution function $F_i$ as the (approximated) conditional distribution of
$\log(Z_i(t_0)) \mid Z_i(\bt)$, $i=1,\ldots,K$, and calculate 
$CRPS(F_i,\log(Z_i(t_0)))$ via \eqref{eq:CRPS}. Here, we do the 
$\log$-transformation to Gumbel marginals to ensure that the conditional
distribution has finite expectation.
\medskip

Then, a measure for the goodness-of-fit is given by the \emph{mean score} 
\citep{gneiting-raftery-2007}
$$\textstyle CRPS_{K} = \frac 1 K \sum_{i=1}^K CRPS(F_i,\log(Z_i(t_0))).$$
Further, we have a look at the mean absolute error of the conditional 
median
$$ \textstyle  MAE_{K} = \frac 1 K \sum_{i=1}^K \left| F_i^{-1}(0.5) - \log(Z_i(t_0)) \right|.$$

For computational reasons, we choose Smith's \citeyearpar{smith-1990} process
with the deterministic shape function 
$$f(x) = \varphi(x) = (2\pi)^{-1/2} \exp\left(-x^2 / 2\right).$$

Furthermore, let $n= 4$, $\bt = (-2,-1,1,2)$ and $t_0=0$. Figure 
\ref{fig:condplot} shows two realizations of $Z(\cdot)$, the first one is 
sampled unconditionally and the second one is based on conditional
sampling of the first one.
\medskip

\begin{figure}
 \centering \includegraphics[height=5.5cm,width=11cm]{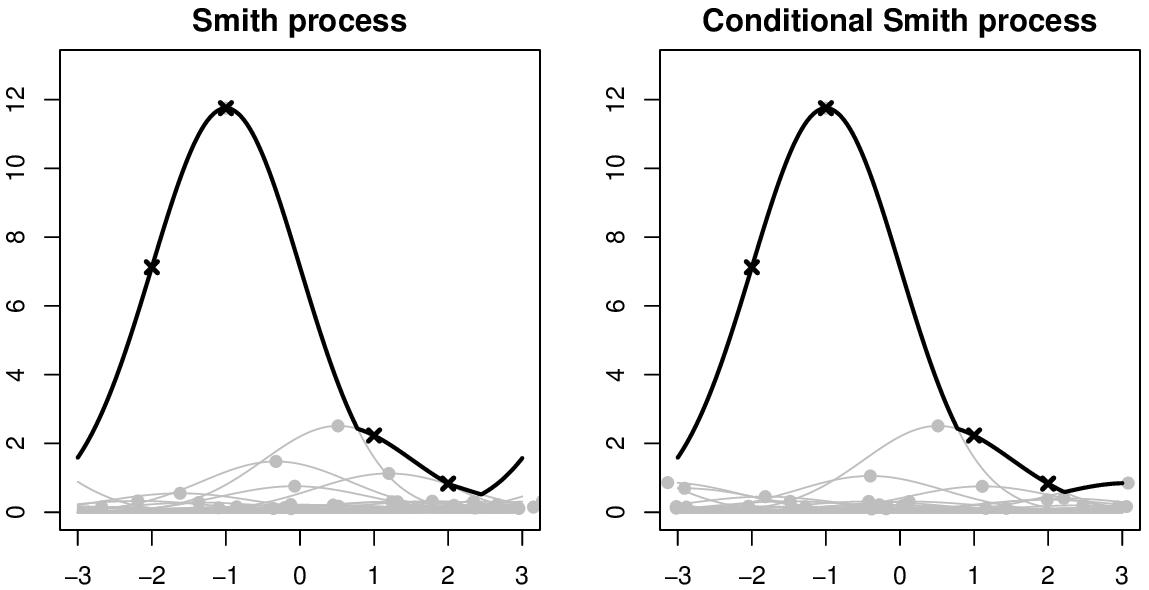}
 \caption{Left: Construction of $Z$. The grey dots represent the points 
          $(s,u \cdot f(0))$ with $(s,u,f) \in \Pi$, the black line is one
          realization of $Z$. The black dots mark $Z(\bt)$.
          Right: Construction of $Z$ conditional on $Z(\bt)$.} \label{fig:condplot}
 \vspace{-0.3cm}
\end{figure}

The conditional distribution is calculated based on a sample of size $100$ 
simulated in \texttt{R} \citep{ihaka-gentleman-1996}. The performance is
measured via $CRPS_{K}$ and $MAE_{K}$ for the methods $PPP$ (conditional 
sampling via the Poisson point process), $MLM$ (conditional sampling for a
max-linear model with $M=5$ and $h=0.1$ using the \texttt{R} package 
\texttt{maxLinear} \citep{maxlinear-2010})  and $GPT$ (conditional sampling
of a Gaussian process with transformed Fr\'echet marginals) with $K=1000$ 
samples. As already mentioned, the last approach requires the knowledge of the
covariance structure of the transformed random field. This is assessed by first
simulating data from this model on a dense grid repeatedly and then estimating
the parameters of a Whittle-Mat\'ern covariance model based on maximum
likelihood techniques implemented in the \texttt{R} package 
\texttt{RandomFields} \citep{randomfields}.

The parameters are chosen such that of the first and second method have a 
similar running time. For these parameters, $GPT$ runs much faster than $PPP$
and $MLM$.
In general, however, the running times scale differently in the number $n$ of 
observations as well as in the number of shape functions. The running time of
$PPP$ grows linearly in the number of shape functions and exponentially in $n$.
Making use of some conditional independence structure, \citet{wang-stoev-2010}
could improve the complexity of their algorithm.  Thus, the running time of
$MLM$ depends linearly on both $p$ (which is a multiple of the number of shape
functions) and $n$ as \citet{wang-stoev-2010} report. The complexity of the
last method, $GPT$, only depends on $n$ and is of order $\OO(n^3)$ as the
conditional expectation and variance of the (marginally) Gaussian distribution
has to be calculated.

\begin{table}
 \begin{center}
\caption{Results of the simulation study for $f(x)=\varphi(x)$ and $K=1000$.}
\label{res}
\begin{tabular}{lrrr}
\hline\noalign{\smallskip}
& $PPP$ & $MLM$ & $GPT$ \\ 
\noalign{\smallskip}\hline\noalign{\smallskip}
$CRPS_{K}$ & -0.135 & -0.359 & -0.251\\ 
$MAE_{K}$  &  0.197 &  0.506 &  0.338\\ 
\noalign{\smallskip}\hline
\end{tabular}
\end{center}
\vspace{-0.3cm}
\end{table}

The results of the simulation study are shown in Table \ref{res}.
Here, $CRPS_{K}$ and $MAE_{K}$ for $PPP$ can be interpreted as reference values
as the first method is exact. We note that conditional sampling for max-linear
models performs worse than conditional sampling via transformation to Gaussian
marginals.
\medskip

For further analysis and comparison of these methods we do not restrict 
ourselves to pointwise prediction, but have a look at the sample paths.
Additionally, pointwise quantile estimation of the conditional distribution can
be done  including the special case of the conditional median which can be seen
as an analogue to kriging \citep{chiles-delfiner-1999}. In case of conditional 
sampling via the Poisson point process and conditional sampling of a max-linear
model the quantiles have to be estimated from the empirical conditional
distribution. For sampling via Gaussian processes the quantiles can be
calculated from the kriged value and the kriging variance.
\medskip

Figure \ref{fig:condplot_all} shows five sample paths and the median of the
Smith process on
\begin{enumerate}
 \item[a. ] observations at four locations $-2$, $-1$, $1$, $2$,
 \item[b. ] observations at eleven locations $-2.5, -2, \ldots, 2, 2.5$.
\end{enumerate}
In general, conditional simulation via the Poisson point process yields sample
paths which capture the main features of the process quite well. Even in the
case of four observations parts of the sample path are reconstructed exactly
with a positive probability. For eleven observations most of the sample path
is restored with high probability.

The results of conditional sampling of the max-linear model are similar to the
first method in case of four observations. For eleven observations,
however, the method fails because of model misspecification.
In the max-linear model, the points generating the observations are assumed
to be located on a lattice $h\Z^d + \frac 1 2$, not at arbitrary locations in
$\R^d$ as assumed in \eqref{eq:procdef}. Due to this restriction, the data do
not match the max-linear model and some observations cannot be reconstructed.
For some realizations of the Smith process this problem even occurs in
case of four observations. This is the main reason for the unsatisfying results
of this method in the simulation study above. Note that -- as computational
experiments show -- misspecification most often occurs if at least three
observations are generated by the same point. However, for any $h >0$, with
probability one, this point is not in $h\Z^d + \frac 1 2$ and therefore, in 
these cases, conditional sampling from the max-linear model fails even for
small $h$. Thus, although the joint distribution can be approximated 
arbitrarily well as $h \searrow 0$, the problem of misspecification in the
algorithm of \cite{wang-stoev-2010} is not resolved.

Conditional sampling for Gaussian processes with transformed marginals yields
sample paths which are structurally very different from the true ones. However,
for eleven observations the deviations from the original sample path are quite
small.

\begin{figure} \label{fig:condplot-all}

\vspace{-0.212cm}
\begin{center}
\begin{minipage}[t]{0.025\textwidth}
 \vspace{-3.9cm} a.
\end{minipage}
\begin{minipage}{0.72\textwidth}
\centering \includegraphics[height=8.4cm,width=8.4cm]{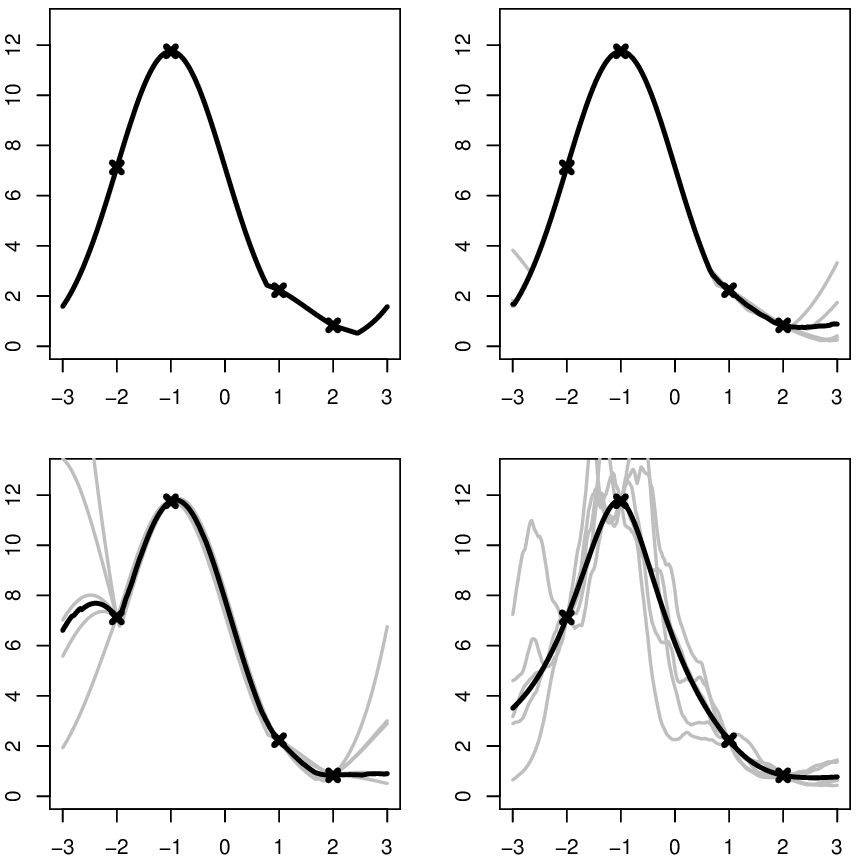}
\end{minipage}

\vspace{-0.212cm}
\begin{minipage}[t]{0.025\textwidth}
 \vspace{-3.9cm} b.
\end{minipage}
\begin{minipage}{0.72\textwidth}
\centering\includegraphics[height=8.4cm,width=8.4cm]{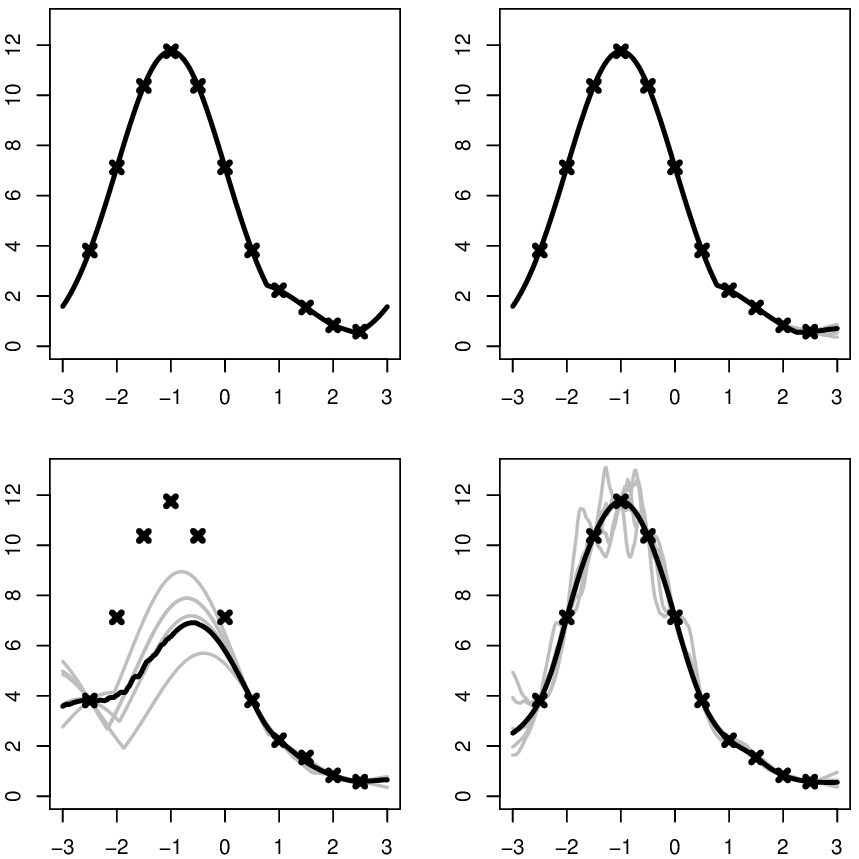}
\end{minipage}
\end{center}
 \caption{Comparison of the Smith process with different types of
          conditional simulations: 
          a. simulations conditional on four observations at $-2$, $-1$, $1$, $2$,
          b. simulations conditional on eleven observations at $-2.5, -2, \ldots, 2, 2.5$.
          In both cases the original Smith process (top left), conditional
          samples via the Poisson point process (top right) and conditional 
          results for a max-linear approximation (bottom left) and an 
          approximation via a Gaussian process with transformed marginals 
          (bottom right) are shown.
          Black crosses: observations, grey lines: conditional sample paths,
          black line: conditional median.} \label{fig:condplot_all}
\end{figure}
\medskip

Finally, we investigate the behaviour of the different algorithms if the 
observations are in the tails of the max-stable distribution. Thus, we repeat
the simulation study above considering $K=1000 $ samples of the random field 
$Z$ conditional on $\min\{Z(t_1), \ldots, Z(t_n)\} \geq \Phi_1^{-1}(q)$. Note
that the covariance structure used for $GPT$ is again estimated from 
transformed samples of the random field $Z$. Besides $GPT$, we also consider an
adjusted version ($AGPT$) where the scale parameter $c$ in \eqref{eq:whittle}
is estimated based on extreme samples 
$(Z(t_1), \ldots, Z(t_n)) \mid \min\{Z(t_1), \ldots, Z(t_n)\} \geq \Phi_1^{-1}(q)$
and the smoothness parameter $\nu$ is the same as for $GPT$. By this 
modification of the scale, we account for possible changes of the covariance
structure in the extremes. The results of the simulation study for $q=0.9$,
$q=0.95$ and $q=0.99$ are shown in Table \ref{res-cond}.
Note that the probability that all the observations $(t_1, Z(t_i))$, 
$i=1,\ldots,n$, are generated by the same point  tends to one as $q \nearrow 1$
and thus the distribution of $Z(t_0)\mid Z(\bt)$ becomes more and more 
concentrated. Therefore, the CRPS and MAE of the exact algorithm ($PPP$) tend 
to zero. For Wang and Stoev's \citeyearpar{wang-stoev-2010} algorithm, however,
the results get worse as $q$ approaches 1. Here, the misspecification issue
gets even more problematic as the probability that all the observations are
generated by the same point increases. Thus, a large number of data does not
match the  max-linear model and the algorithm of \cite{wang-stoev-2010} yields
unsatisfactory results. The CRPS and MAE of the algorithm for Gaussian 
processes with transformed marginals also both get worse as $q$ gets close to
1. This is due to the fact that Gaussian random variables are asymptotically
independent. Thus, the algorithm is not able to capture the joint tail
behaviour well if we use the the same covariance structure as for the non-extreme
observations. For the adjusted version, however, where the modified scale 
parameter leads to stronger correlations, both the CRPS and the MAE improve
as $q$ approaches $1$. 

\begin{table}
 \begin{center}
\caption{Results of the simulation study for the Smith process. CRPS and MAE 
 for the distribution of 
 $Z(t_0) \mid Z(\bt)$ based on $K=1000$ samples conditional on
 $\min Z(\bt) \geq \Phi_1^{-1}(q)$.} \label{res-cond}

\begin{tabular}{lcrrrrcrrrr}
\hline\noalign{\smallskip}
 & &\multicolumn{4}{c}{CRPS}
 & & \multicolumn{4}{c}{MAE}\\ 
 \noalign{\smallskip} \cline{3-6} \cline{8-11} \noalign{\smallskip}
 $q$  & & $PPP$ & $MLM$ & $GPT$ & $AGPT$ & & $PPP$ & $MLM$ & $GPT$ & $AGPT$\\ 
\noalign{\smallskip}\hline\noalign{\smallskip}
0.90 & & -0.014 & -1.227 & -0.338 & -0.234 & & 0.016 & 1.693 & 0.415 & 0.284\\ 
0.95 & & -0.006 & -1.525 & -0.389 & -0.196 & & 0.006 & 2.104 & 0.491 & 0.228\\
0.99 & & -0.001 & -1.901 & -0.568 & -0.126 & & 0.000 & 2.611 & 0.856 & 0.117\\
\noalign{\smallskip}\hline
\end{tabular}
\end{center}
\vspace{-0.4cm}
\end{table}

\section{Approximation in the case of an infinite number of shape functions} 
\label{sec-approx}

Here, we drop the assumption that $G$ is finite. We present an approximation of
the distribution of $Z(\cdot)$ given $Z(\bt)$ based on a finite number of shape
functions. Let $F_1, F_2, \ldots$ be independent copies of $F$ where $F$ is 
defined as in Section \ref{sec-intro}. Then, given $F_1,\ldots, F_N$, we define
\begin{equation}
\textstyle Z_N(t) = \max_{(s,u,f) \in \Pi^{(N)}} u f(t-s), \quad t \in \R^d,
\end{equation}
where $\Pi^{(N)}$ is a Poisson point process on 
$\R^d \times (0,\infty) \times \{F_1,\ldots,F_N\}$ with intensity measure
$$ \Lambda(A \times B \times \{F_k\}) =  \frac 1 N \int_A \int_B u^{-2} \sd u \sd s, \ A \in \BB^d,
 \ B \in \BB \cap (0,\infty), \ k \in \{1\ldots,n\}.$$
For a proof of the following theorem, see \citet[][Ch.~5]{oesting-diss}.

\begin{theorem} \label{approximationthm}
 For any $\bz > \mathbf{0}$ we have
 $$ \Pi^{(N)} \mid Z_N(\bt) \leq \bz \stackrel{D}{\longrightarrow} \Pi \mid Z(\bt) \leq \bz$$
 as $N \to \infty$.
 In particular, $Z_N(\cdot) \mid Z_N(\bt) \leq \bz \stackrel{D}{\longrightarrow} Z(\cdot) \mid Z(\bt) \leq \bz$
 in the sense of finite-dimensional distributions.
\end{theorem}
\medskip

If $G$ is countable, we apply the second part of Corollary 
\ref{condindependence} to the process $Z_N$ yielding
$ \lim_{N \to \infty} \Pi_3^{(N)} \mid (Z_N(\bt) = \bz) 
\stackrel{D}{=} \lim_{N \to \infty} \left(\Pi_N \setminus (K_{\bt,\bz} \cup \overline{K_{\bt,\bz}})\right).$
Thus, Theorem \ref{approximationthm} and the second part of Corollary 
\ref{condindependence} (applied to $Z$) imply that
\begin{align*}
 \lim_{N \to \infty} \Pi_3^{(N)} \mid (Z_N(\bt) = \bz)
\stackrel{D}{=}{} & \Pi \setminus (K_{\bt,\bz} \cup \overline{K_{\bt,\bz}}) {}\stackrel{D}{=}{} \Pi_3 \mid (Z(\bt)=\bz).
\end{align*}
This motivates to improve the approximation $\Pi \approx \Pi^{(N)}$ by
$\Pi \approx \Pi_2^{(N)} \cup \Pi_3$, i.e.\ by the following procedure:
 \begin{enumerate}
  \item Simulate $\Pi^{(N)}_2 \mid Z(\bt)= \bz$.
  \item Independently of $\Pi_2^{(N)}$, sample $\Pi_3 \mid Z(\bt)=\bz$
    defined as $\Pi \cap (\R \times (0,\infty) \times G) \setminus (K_{\bt,\bz} \cup \overline{K_{\bt,\bz}})$
    analogously to the second part of Corollary \ref{condindependence}.
 \end{enumerate}
 Then, $Z(\cdot) \approx \max_{(s,u,f) \in \Pi_2 \cup \Pi_3} u f(\cdot-s)$.
\medskip

\section{Application to the Brown-Resnick process} \label{sec-BR}

We will apply the method of conditional sampling via the Poisson point process
to the process constructed by \cite{brown-resnick-1977}.

Let $\{W_x(t),\ t \in \R\}$, $x \in (0,\infty)$, be independent copies of a
standard Brownian motion and --- independently of the $W_x$'s --- let 
$\tilde \Pi$  be a Poisson point process on $(0,\infty)$ with intensity measure
$u^{-2} \sd u$. Then,
\begin{equation} \label{eq:BR-def}
 \textstyle Z(t) = \max_{u \in \tilde \Pi} \left(u \exp\left(W_u(t) - |t| / 2\right)\right), \quad t \in \R,
\end{equation}
defines a stationary max-stable process with standard Fr\'echet margins.

Recently, this process was generalized \citep{kab-etal-2009} yielding a class 
of processes which essentially corresponds to the class of processes that occur
as the limit of maxima of independent Gaussian processes 
\citep[cf.][]{kabluchko-2011}. Some of these processes also allow for a M3
representation \eqref{eq:procdef}. In the general case, the shape function has
so far only been expressed implicitly as a conditional distribution depending 
on the point process of the original construction \citep[][]{oesting-etal-2012}.
In case of the original Brown-Resnick process, however, it can be given 
explicitly \citep{engelke-2011}. In particular,
\begin{equation} \label{eq:seb} \textstyle
 Z(t) \stackrel{D}{=} \max_{(s,u,f) \in \Pi} \left(u / 2 \cdot \exp(-f(t-s))\right), \quad t \in \R,
\end{equation}
where $\Pi$ is a Poisson point process on $\R \times (0,\infty) \times C(\R)$
with intensity measure $\sd s \ u^{-2} \sd u \ \Prob_R({\rm d}f)$ and $\Prob_R$
is the law of the  process
$$R(t) = \mathbf{1}_{t<0} R_1(-t) + \mathbf{1}_{t \geq 0} R_2(t).$$
Here, $\{R_1(t), \ t \geq 0\}$, $\{R_2(t), \ t > 0\}$ are independent Bessel 
processes of a three-dimensional Brownian motion with drift $\frac 1 2$ in its
first component  \citep[cf.][]{rogers-pitman-1981}, i.e.\
$$R_1(t) \stackrel{D}{=} R_2(t) \stackrel{D}{=} \sqrt{(W_1(t) + |t|/2)^2 + W_2(t)^2 + W_3(t)^2},$$
where $W_1, W_2$ and $W_3$ are independent standard Brownian motions.

We will use the results obtained in the section above to sample from the 
conditional distribution of the Brown-Resnick process. However, the sample
paths of $\exp(-R(\cdot))$ do not satisfy the assumptions of Propositions
 \ref{intersectintens}, \ref{3points} and \ref{Kintens}. In particular, 
 with probability one, the sample paths are not differentiable anywhere.
To overcome this drawback, we do not use the exact sample paths 
$F(\cdot) = \exp(-R(\cdot))$, but the sample paths evaluated on a grid and
interpolated linearly in between. Thus, sample path properties like
differentiability are changed. However, for small mesh width, the difference
to the original sample path should be invisible.

Let $T = \{t_z, \ z \in \Z\} \subset \R$ with 
$\ldots < t_{-2} < t_{-1} < t_0 < t_1 < t_2 < \ldots$ such that 
$\lim_{z \to -\infty} t_z = -\infty$, $\lim_{z \to \infty} t_z = \infty$ and
$||T|| = \sup_{z \in \Z} (t_z - t_{z-1})$. Let $\{(t,F_T(t)), \ t \in \R\}$ be
the polygonal line through the points $\{(t,F(t)), \ t \in T\}$. Furthermore,
define $Z_T(t)$ as in \eqref{eq:seb}, replacing $f$ by $F_T$. Then, for $||T|| \to 0$,
we have $Z_T(t) \to Z(t)$ in probability for all $t \in \R$.
 In particular,
$$Z_T(\cdot) \mid (Z_T(\bt) \in B) \stackrel{D}{\longrightarrow} Z(\cdot) \mid (Z(\bt) \in B)$$
in the sense of finite-dimensional distributions for all Borel sets 
$B \subset \R^n$ with $\Prob(Z(\bt) \in B) > 0$ and 
$\Prob(Z(\bt) \in \partial B) = 0$ \citet[cf.][Ch.~5]{oesting-diss}.
Thus, $Z$ can be approximated arbitrarily well by $Z_T$. 
However, still, for any fixed $T$, the range of $\frac 1 2 F_T$
is uncountable. Therefore, we have to use the approximation introduced in 
Section \ref{sec-approx}.
\medskip

We compare this approximation to conditional sampling based on the approach of
\cite{wang-stoev-2010}, the approach via a Gaussian process with transformed 
marginals, see Section \ref{sec-compare}, and the exact algorithm of 
\citet{dombry-ribatet-2012} and \citet[Ch.~6]{oesting-diss} for 
conditional sampling of Brown-Resnick processes. The basic steps of the exact
algorithm are the same as in our Poisson point process approach. However,
it is based on the original representation \eqref{eq:BR-def} instead of the
equivalent M3 representation. As the exponent measure of the
Brown-Resnick process is absolutely continuous w.r.t.\ the Lebesgue measure,
the results of \cite{dombry-2011} provide explicit formulae allowing for exact
conditional simulation in this case.
\medskip

To compare these procedures, we simulate $K=500$ independent samples of $Z$ on 
the set $\{t_0, t_1, t_2, t_3, t_4\}$ with $t_0 = 0$, $t_1=-2$, $t_2 = -1$,
$t_3=1$ and $t_4=2$. We calculate the CRPS and MAE by sampling $100$ times
from the (approximate) conditional distribution of $\log(Z(t_0))$ given 
$Z(\bt)$. Besides two variants of conditional sampling of the Poisson point
process, which we will denote by $PPP_1$ and $PPP_2$, let $MLM$ denote the
approach by \cite{wang-stoev-2010}, $GPT$ the algorithm based on a Gaussian
process with transformed Fr\'echet marginals and $BR$ the exact algorithm by
\cite{dombry-ribatet-2012}. 
For the Poisson point process approach, we chose $N=250$ as the number of shape
functions on the grid $T=\{-5,-4.9,\ldots,4.9,5\}$. However, if we restrict
ourselves to a finite number of shape functions, the intersection set 
$I_{A}(\bz)$ with $|A|\geq 3$ is most likely empty, even though $Z(t_i)$, 
$i \in A$, may be determined by the same $(s,u,f) \in \Pi$. Therefore, we do
not only consider ``exact'' intersections, but also intersections which occur
if the function values differ up to a given tolerance, i.e.\ we assume 
$(t,y,f) \in I_A(\bz)$ with $y = \min_{i=1,\ldots,n} \frac {z_i}{f(t_i-t)}$ if
$$ \frac{z_i}{f(t_i-t)} < \min\left\{ y + \varepsilon, y (1+\varepsilon)\right\}
 \quad \Longleftrightarrow \quad i \in A.$$
for some given tolerance $\varepsilon > 0$. The simulation study is
performed for $\varepsilon = 10^{-6}$ ($PPP_1$) and 
$\varepsilon = 10^{-2}$ ($PPP_2$).
By these choices, $PPP_1$ practically excludes intersections of more than
two curves, while these still occur in $PPP_2$.
For the $MLM$ approach, we use the same approximation technique as in
Section \ref{sec-compare}. Here, we chose $M=5$, $h=0.1$ and the same $N=250$
shape functions as for the Poisson point process approach. The parameters are
chosen such that $PPP_1$, $PPP_2$ and $MLM$ have similar running times. We
observe that all the methods have a similar accuracy except for $PPP_2$. 
However, $GPT$ runs much faster than the others. Detailed results are displayed
in Table \ref{res-BR}, where, again, $CRPS_{K}$ and $MAE_{K}$ for $BR$ can be
interpreted as reference values as this method is exact.

\begin{table}
 \begin{center}
\caption{Simulated results for the Brown-Resnick process with 
         $N=250$ and $K=500$.} \label{res-BR}
\begin{tabular}{lrrrrr}
\hline\noalign{\smallskip}
& $PPP_1$ & $PPP_2$ & $MLM$ & $GPT$ & $BR$ \\
\noalign{\smallskip}\hline\noalign{\smallskip}
$CRPS_{K}$ & -0.366 & -0.493 & -0.381 & -0.364 & -0.355\\
$MAE_{K}$  &  0.513 & 0.606 & 0.515 & 0.513 & 0.504\\
\noalign{\smallskip}\hline
\end{tabular}
\end{center}
\vspace{-0.3cm}
\end{table}

Note that, here, $MLM$ performs slightly worse than $PPP_1$ and $GPT$. 
However, $MLM$ is competitive for $h$ small enough and $N$ large enough.
Furthermore, we notice the difference between $PPP_1$ and $PPP_2$ indicating
that considering approximate intersections of at least three curves yields
worse results. This is because these intersections involve incorrect shape
functions. Furthermore, intersections of three curves lead to degenerated
conditional distributions which are not supposed to occur in the case of the
Brown-Resnick process. Thus, $PPP_2$ seems to be an inappropriate procedure in
this case.
\medskip

Analogously to Section \ref{sec-compare}, we also compare the behaviour of the
algorithms when the observations from the Brown-Resnick process are extreme.
Thus, we repeat the simulation above with $K=500$ samples from the 
Brown-Resnick process conditional on 
$\min\{Z(t_1), \ldots, Z(t_n)\} \geq \Phi_1^{-1}(q)$ and apply the algorithms
$PPP_1$, $MLM$, $GPT$, $AGPT$ and $BR$ to draw (approximately) from the distribution of
$ Z(t_0) \mid Z(\bt)$. Note that we do not use the algorithm $PPP_2$ here, as 
it turned out to be inappropriate in the non-extreme case. The results 
for $q=0.9$, $q=0.95$ and $q=0.99$ are shown in Table \ref{res-BR-cond}.

\begin{table}
 \begin{center}
\caption{Simulated results for the Brown-Resnick process ($N=250$). CRPS and
MAE for the distribution of $Z(t_0) \mid Z(\bt)$ based on $K=500$ samples
conditional on $\min Z(\bt) \geq \Phi_1^{-1}(q)$.} 
\label{res-BR-cond}

\begin{tabular}{lcrrrrr}
\hline\noalign{\smallskip}
& &\multicolumn{5}{c}{CRPS}\\ \noalign{\smallskip} \cline{3-7} \noalign{\smallskip}
$q$ & & $PPP_1$ & $MLM$ & $GPT$ & $AGPT$ & $BR$\\ 
\noalign{\smallskip}\hline\noalign{\smallskip}
0.90 & & -0.404 & -0.392 & -0.493 & -0.379 & -0.370\\ 
0.95 & & -0.452 & -0.435 & -0.538 & -0.416 & -0.416\\
0.99 & & -0.454 & -0.436 & -0.596 & -0.423 & -0.415\\
\noalign{\smallskip}\hline
\end{tabular}
\medskip 

\begin{tabular}{lcrrrrr}
\hline\noalign{\smallskip}
& &\multicolumn{5}{c}{MAE}\\ \noalign{\smallskip} \cline{3-7} \noalign{\smallskip}
$q$ & & $PPP_1$ & $MLM$ & $GPT$ & $AGPT$ & $BR$\\ 
\noalign{\smallskip}\hline\noalign{\smallskip}
0.90 & & 0.570 & 0.549 & 0.622 & 0.526 & 0.523\\ 
0.95 & & 0.641 & 0.613 & 0.693 & 0.586 & 0.592\\
0.99 & & 0.637 & 0.607 & 0.763 & 0.586 & 0.579\\
\noalign{\smallskip}\hline
\end{tabular}

\end{center}
\vspace{-0.3cm}
\end{table}

When conditioning on extreme observations, the simulation results depict more
clearly that the $BR$ algorithm is exact while $PPP_1$, $MLM$ and $GPT$ only
yield approximations to the conditional distribution. 
Among these three, the algorithm for max-linear models by \cite{wang-stoev-2010}
performs best. Note that its results can be improved further by decreasing $h$
and increasing $M$ and $N$. Here, the misspecification problem can be neglected
if  $N$ is large enough, as the as the support of the density of the shape 
function covers the whole space.
The point process based approach $PPP_1$ performs slightly worse than $MLM$.
One may conclude that the approximation of the non-differentiable shape 
functions by a finite number of polygonal lines is less accurate for extreme
observations. Furthermore, similarly to the case of Smith's 
\citeyearpar{smith-1990} process, due to the asymptotic independence of
Gaussian random variables, the results for the algorithm $GPT$ for Gaussian
processes with transformed marginals get worse as $q \nearrow 1$. However,
again, the results improve remarkably if the covariance structure is adjusted
($AGPT$). Thus, the algorithm for Gaussian processes becomes competitive to 
the other algorithms even in the case of extreme observations.
\medskip

\section{The discretized case} \label{sec-discrete}

By now, we have considered the general model \eqref{eq:procdef}. The procedure
we proposed is exact in the case of a finite number of shape functions which
are sufficiently smooth.  However, as the example of the Brown-Resnick process
in Section \ref{sec-BR} illustrates, we may run into problems if these 
assumptions are violated.

Now, we modify our general model \eqref{eq:procdef} and use a discretized
version
\begin{equation} \textstyle
Z(t) = \max_{(s,u,f) \in \Pi} u f(t-s), \quad t \in p\Z^d, \label{eq:procdef-discrete}
\end{equation}
where $\Pi$ is a Poisson point process on $p\Z^d \times (0,\infty) \times G$ 
where $p > 0$ and $G \subset [0,\infty)^{p\Z^d}$ is countable. The intensity
measure of $\Pi$ is given by 
$$\textstyle \Lambda(\{s\} \times B \times \{g\}) = \sum_{z \in \Z^d} \delta_{pz}({\rm d} s)
\times \int_B u^{-2} \sd u \times \Prob_F(\{g\})$$
where $\Prob_F$ is the push forward measure of a $G$-valued random variable 
$F$ with $\E(\sum_{z \in \Z^d} F(pz)) = 1$.
\medskip

Using the same notations as before, we obtain the same results as in Section
\ref{sec-limits}. However, all the calculations can be done explicitly without
any further assumptions on $f \in G$. We get the following results.

\begin{proposition} \label{single-set-discrete}
Let $i \in \{1,\ldots,n\}$, $\bz > \mathbf 0$ and
$$D_i(\bz) = \{ (x,f) \in p\Z^d \times G: \ (x,y,f) \in I_{\{i\}}(\bz) \textit{ for some } y \in \R\}.$$
Then, we have
$$ \textstyle \Lambda( I^{(m)}_{\{i\}}(\bz)) = 
  2^{-m} \sum_{(x,f) \in D_i(\bz)} \frac{f(t_i-x)}{z_i} \Prob_F(\{f\}) + o(2^{-m}).$$
\end{proposition}
\medskip

\begin{proposition} \label{intersection-intens-discrete}
 Let $A \in \AA$, $|A|>1$ and $\bz > \mathbf 0$ such that
 $ l(\bz) = |I_A(\bz)| > 0$. In particular, let
 $I_A(\bz) = \{ (x_j,y_j,f_j), \ j=1,\ldots,l(\bz)\}$.
 Then, for $m$ large enough, we have
 \begin{align*} 
\Lambda(I_A^{(m)}(\bz)) ={} & \textstyle \sum_{j=1}^{l(\bz)} \frac 1 {y_j} \Prob_F(\{f_j\}) \cdot \left( \bigwedge_{i \in A}
\frac {2^m z_i}{j_m(z_i)} - \bigvee_{i \in A} \frac{2^m z_i}{j_m(z_i)+1}\right) 
 \end{align*}
 Thus, $\Lambda(I_A^{(m)}(\bz)) \in \OO(2^{-m})$, but 
 $\Lambda(I_A^{(m)}(\bz)) \notin \OO(2^{-m(1+\varepsilon)})$ for any 
 $\varepsilon > 0$.
\end{proposition}

By these formulae, all the scenario probabilities can be calculated. As the 
intensity of each intersection set has the same rate of convergence, only
scenarios with minimal $|\Pi \cap K_{\bt,\bz}|$ occur.
\medskip

We note that our model is very close to the model investigated by 
\cite{wang-stoev-2010}. To see this, we calculate that
$$ \textstyle \Prob(Z(\bt) \leq \bz)
  = \exp\left(- \sum_{f \in G} \sum_{m \in \Z^d} \bigvee_{i=1}^n \frac{f(t_i-pm) \Prob_F(\{f\})}{z_i}\right).$$
Therefore, we get that
$$ \textstyle Z \stackrel{D}{=} \max_{z \in \Z^d} \max_{f \in F} \big(f(\cdot-pz)\Prob_F(\{f\}) Z_f^{(z)}\big),$$
where the random variables $Z_f^{(z)}$, $z \in \Z^d$, $f \in G$, are
independently standard Fr\'echet distributed.

This means, the model \eqref{eq:procdef-discrete} is a max-linear model if $G$
is finite and the support of each $f \in G$ is finite. In this special case 
both the algorithm of conditional sampling of the Poisson point process and the
algorithm of \cite{wang-stoev-2010} provide the exact conditional distribution,
which is confirmed by computational experiments in case of data from a 
discretized model \eqref{eq:procdef-discrete}. For data from a continuous
M3 process \eqref{eq:procdef}, both algorithms fail because of
model misspecification (cf.\ Section \ref{sec-compare}).
However, both algorithms do not work in exactly the same way.
According to the algorithm of \citet{wang-stoev-2010}, one samples from each
random variable $Z_f^{(z)}$. This procedure corresponds to simulating the 
largest point of $ \Pi \cap (0,\infty) \times \{pz\} \times \{f\}$ for each
$z \in \Z$, $f \in G$. The point-process-based algorithm includes the 
simulation of points in $\Pi$ until a terminating condition given in Theorem 4
of \cite{schlather-2002} is met.
\medskip

Despite of the different approaches, also technical results provided in this
section are related to the ones in \cite{wang-stoev-2010}. For example, the
occurrence of a scenario $J \subset \Z^d \times G$ 
(in the notation of Wang/Stoev) corresponds to the event that $\Pi_2$ consists
of $|J|$ elements $(pz, \min_{i =1,\ldots,n} \frac{Z(t_i)}{f(t_i-pz)}, f)$ with
$(z,f) \in J$. By this correspondence, the statements
\begin{itemize}
 \item $|\Pi \cap K_{\bt,\bz}|$ is minimal a.s.
 \item an occurring hitting scenario $J$ satisfies 
       $|J| = r(\mathcal{J}(A,\bx))$ a.s. (\cite{wang-stoev-2010})
\end{itemize}
are equivalent, both claiming that the number of points generating the 
observation $(\bt,\bz)$ is minimal. Hence, in spite of different approaches,
there are similar observations and results in \cite{wang-stoev-2010} and in
this section.
\medskip

\section{Summary and Discussion} \label{sec-discussion}

The theoretical results together with the simulation studies allow for a
comprehensive picture of the different algorithms with their positive and
negative aspects.

The Poisson point process based approach presented in this paper provides
exact conditional distributions for M3 processes with a
finite number of sufficiently smooth shape functions on the real line. 
Approximations are proposed if the conditions on the shape functions are not
met. They seem to work quite well in case of the Brown-Resnick process, in 
general. However, they might be inaccurate for extreme observations. As the 
number of scenarios with a positive probability might increase exponentially 
\citep[cf.][Example 5.18]{oesting-diss} in the number $n$ of observations, so
does the running time.

\cite{wang-stoev-2010} provide an exact and efficient algorithm for max-linear
models, that scales linearly in $n$. Although any multivariate max-stable
distribution can be approximated arbitrarily well by a max-linear model,
data stemming from a non-regular M3 process (e.g.\ the Smith
process) may lead to a misspecification problem independently from the
quality of approximation.

Conditional sampling via Gaussian processes with transformed marginals
is exact only for max-stable processes with Gaussian dependence structure.
In case of regular models like the Brown-Resnick process the algorithm
works quite well in general. However, using the overall covariance 
structure, it fails to capture the dependence structure well in case of
extreme observations.

If the covariance structure is estimated from extreme observations only,
the results for this case are surprisingly good. For large and moderate
numbers of observations, the running time of the algorithm for Gaussian
processes is much faster than the one of the other algorithms. In general,
it is of the order of $n^3$.

\cite{dombry-2011} give formulae for the conditional distribution of 
any max-stable process in terms of the exponent measure. These formulae
are directly applicable only if the exponent measure is absolutely continuous
w.r.t.\ the Lebesgue measure as in the case of Brown-Resnick or extremal
Gaussian processes \citep[cf.][]{dombry-ribatet-2012}. As it involves all
partitions of the set $\{1, \ldots, n\}$, the calculation of the exact
conditional distribution is of the same order as the Bell numbers
which grow super-exponentially. \cite{dombry-ribatet-2012} propose MCMC
methods to reduce the computational burden.

In general, our results indicate that, at least in some regular cases and 
w.r.t.\ the CRPS, the algorithm for Gaussian processes, appropriately adjusted
in case of extreme observations, might be a very attractive alternative to more
accurate but also more complicated Poisson point process based methods.

\section{Calculations in the case of a finite number of shape functions on the 
         real line} \label{sec-calculate}

This section contains the proofs of the Propositions \ref{intersectintens},
\ref{3points}, \ref{more-3points}, \ref{Kintens} and Lemma 
\ref{neglectprobs}, providing the explicit calculations of the intensities.

\def\proofname{Proof of Proposition \ref{intersectintens}}

\begin{proof}
We note that $(t_0,y_0,f)$ satisfies the equation
$$ f(t_1-t_0) / z_1 = f(t_2-t_0) /z_2 = y_0^{-1}.$$
Let
$$ H: (-\bz,\infty) \times \R \to \R, \ (\delta,t) \mapsto \frac{f(t_1-t)}{z_1+\delta_1}
- \frac{f(t_2-t)}{z_2+\delta_2}.$$
Then, $H(\mathbf{0},t_0) = 0$ and
$$\frac{\partial H}{\partial t}(\mathbf{0},t_0) = - \frac{f'(t_1-t_0)}{z_1} + \frac{f'(t_2-t_0)}{z_2} \neq 0$$
due to \eqref{eq:invertibilitycrit}.
The implicit function theorem yields the existence of a neighbourhood $V$ of
$\mathbf{0}$ and a continuously differentiable function
$h: V \to \R$ such that $H(\delta,h(\delta)) = 0$. Using the notation
$(t_{\delta},y_{\delta},f) = I_{\{1,2\}}(\bz +\delta)$ we get 
$h(\delta) = t_{\delta}$ and the equality
$$\frac {f(t_1-t_{\delta})}{z_1+\delta_1} = \frac{f(t_2-t_{\delta})}{z_2+\delta_2} = y_{\delta}^{-1}.$$
\medskip
As $h$ is $C^1$, we obtain $ t_0 - t_{\delta} \in \OO(||\delta||)$,
and a Taylor expansion of $f$ yields
\begin{equation} \label{eq:taylor}
 f(t_i-t_{\delta}) = f(t_i-t_0) - f'(t_i-t_0) \cdot (t_{\delta} - t_0) + o(||\delta||), \qquad i=1,2.
\end{equation}
Let $g(t) = f(t-t_0)$.
Then, using \eqref{eq:taylor}, $t_{\delta}$ is given implicitly by
$$ \frac{g(t_1) - g'(t_1) \cdot (t_{\delta}-t_0)}{z_1+\delta_1} =
   \frac{g(t_2) - g'(t_2) \cdot (t_{\delta}-t_0)}{z_2+\delta_2} + o(||\delta||),$$
which implies the explicit representation
\begin{equation} \label{eq:tdelta}
t_{\delta} = t_0 + \frac{\delta_1 g(t_2) - \delta_2 g(t_1)}{z_1 g'(t_2) - z_2 g'(t_1)} + o(||\delta||).
\end{equation}
Plugging in \eqref{eq:tdelta} into \eqref{eq:taylor} yields
\begin{equation*}
 y_{\delta}^{-1} = \frac{f(t_1-t_\delta)}{z_1+\delta_1} = \frac{g(t_1)}{z_1+\delta_1} - \frac{g'(t_1)}{z_1 + \delta_1} \cdot 
\frac{\delta_1 g(t_2) - \delta_2 g(t_1)}{z_1 g'(t_2) - z_2 g'(t_1)} + o(||\delta||).
\end{equation*}
As $f$ and $\delta \mapsto t_\delta = h(\delta)$ are $C^1$-functions, all the 
terms $o(||\delta||)$ are continuously differentiable for small $||\delta||$.
Therefore, the mapping
$\Phi: V \to \R \times (0,\infty),$ $\delta \mapsto (t_{\delta},y_{\delta}^{-1})$
is continuously differentiable near the origin.
Calculating the partial derivatives explicitly we obtain
\begin{equation}\label{eq:detjacobian}
 \det(D\Phi(\delta)) = - \frac{g^2(t_1)}{z_1^2\cdot (z_1 g'(t_2) - z_2 g'(t_1))} + o(1).
\end{equation}
As $\det(D\Phi(\mathbf 0)) \neq 0$, the inverse function theorem allows to 
regard $\Phi$ as a diffeomorphism restricted to a neighbourhood of
$\mathbf{0}$. Thus, considering the Poisson point process
$\tilde \Pi = \sum_{(s,u) \in \Pi} \delta_{(s,u^{-1})}$ on 
$\R \times (0,\infty)$ whose intensity measure is the Lebesgue measure, with
$A_m^{(i)} = A_m(z_i) - z_i$ for $i=1,2$, we get
\begin{align*}
 & \Lambda(\{I_{\{1,2\}}(\bz+\delta), \ \delta_i \in A_m^{(i)}, \ i=1,2\})
={}  \int_{\Phi(A_m^{(1)} \times A_m^{(2)})} \Prob_F(\{f\}) \, \sd(t,y) \displaybreak[0]\\
={} & \int_{A_m^{(1)} \times A_m^{(2)}} |\det(D\Phi(\delta))| \cdot \Prob_F(\{f\}) \sd \delta\\
={} & \int_{A_m^{(1)} \times A_m^{(2)}} \frac{1/y_0^2 + o(1)}{|z_1 g'(t_2) - z_2 g'(t_1)|} \Prob_F(\{f\}) \sd \delta.
\end{align*}
We note that the term $o(1)$ is continuous w.r.t.\ $\delta$ and therefore the
integrand can be locally bounded by the interval
$$ \left[\frac{\Prob_F(\{f\})/y_0^2}{|z_1 g'(t_2) - z_2 g'(t_1)|}  - \varepsilon_m,
   \frac{\Prob_F(\{f\})/y_0^2}{|z_1 g'(t_2) - z_2 g'(t_1)|}  + \varepsilon_m\right]$$
for all $(\delta_1,\delta_2) \in A_m^{(1)} \times A_m^{(2)}$ with $m$ large 
enough and an appropriate sequence $(\varepsilon_m)_{m\in\N}$ with 
$\varepsilon_m \searrow 0$. This implies that the integral has the desired form. 
\qed
\end{proof}

\def\proofname{Proof of Proposition \ref{3points}}

\begin{proof}
The first assertion follows immediately from Proposition \ref{intersectintens}
by the fact that 
$ \textstyle \bigcap_{i=1}^l K^{(m)}_{t_i,z_i} \subset K^{(m)}_{t_1,z_1} \cap K^{(m)}_{t_2,z_2}.$

In order to verify the second assertion, we recall results from the proof of
Proposition \ref{intersectintens}: we showed the existence of a $C^1$-function
$(\delta_1, \delta_2) \mapsto t_{\delta_1,\delta_2}$ defined in a neighbourhood
$V$ of $(0,0)$ such that
$\frac{f(t_1-t_{\delta_1,\delta_2})}{z_1+\delta_1} = \frac{f(t_2-t_{\delta_1,\delta_2})}{z_2+\delta_2}.$
Now, for $i \in \{3,\ldots,l\}$, we consider the $C^1$-functions
$$\textstyle H_i: V \times (-z_i,\infty) \to \R, \ (\delta_1,\delta_2,\delta_i) \mapsto \frac{f(t_1-t_{\delta_1,\delta_2})}{z_1+\delta_1}
-\frac{f(t_i-t_{\delta_1,\delta_2})}{z_i+\delta_i}.$$
As $H_i(0,0,0) = 0$ and 
$\frac{\partial H_i}{\partial \delta_i}(0,0,0) = \frac{f(t_i-t_0)}{z_i^2} \neq 0$,
we get the existence of a $C^1$-function $h_i$ defined on a neighbourhood of 
$(0,0)$ such that
\begin{equation} \label{eq:implicit-eq}
\frac{f(t_1-t_{\delta_1,\delta_2})}{z_1+\delta_1}=\frac{f(t_i-t_{\delta_1,\delta_2})}{z_i+h_i(\delta_1,\delta_2)}.
\end{equation}

Using Taylor expansions of $g(\cdot) = f(\cdot-t_0)$ of first order, employing
Equation \eqref{eq:tdelta}, and solving Equation \eqref{eq:implicit-eq} yields
\begin{equation} \label{deltai}
 h_i(\delta_1,\delta_2) = \frac {g(t_i)}{g(t_1)} \delta_1 + \frac{z_i g'(t_1) - z_1 g'(t_i)}{g(t_1)} 
 \frac{g(t_2) \delta_1 - g(t_1) \delta_2} {z_1 g'(t_2) - z_2 g'(t_1)} + o(|\delta_1|) + o(|\delta_2|).
\end{equation}
So, there are constants $c_{1,i}, c_{2,i}$ such that $h_i(\delta_1,\delta_2) =
c_{1,i} \delta_1 +  c_{2,i} \delta_2 + o(|\delta_1|)+o(|\delta_2|)$. Let 
$A_m^{(i)} = A_m(z_i) - z_i$ for $i \in \{1,\ldots,n\}$. We are interested in
those pairs $(\delta_1,\delta_2) \in A_m^{(1)} \times A_m^{(2)}$ with 
$h_i(\delta_1,\delta_2) \in A_m^{(i)}$. By Lemma \ref{lambda-bound}, for any
$C'>0$, $\varepsilon > 0$ and $m$ large enough, we have that 
$(-C'2^{-m(1+\varepsilon)}, C'2^{-m(1+\varepsilon)}) \in A_m^{(i)}$, 
$i=1,\ldots,n$. Therefore, $h_i(\delta_1,\delta_2) \in A_m^{(i)}$ is guaranteed
for $|\delta_1| < \frac{C' 2^{-m(1+\varepsilon)}}{3 |c_{1,i}|}$ and
$|\delta_2| < \frac{C' 2^{-m(1+\varepsilon)}}{3 |c_{2,i}|}$ if $m$ is 
sufficiently large.

By the same argumentation for all $i \in \{3,\ldots,l\}$ we get that the
existence of all $h_i(\delta_1,\delta_2)$ is ensured for 
\begin{equation} \label{eq:delta-assessment}
|\delta_1| < \frac{C' 2^{-m(1+\varepsilon)}}{3\max_{i=3,\ldots,l} |c_{1,i}|}, \quad
|\delta_2| < \frac{C' 2^{-m(1+\varepsilon)}}{3\max_{i=3,\ldots,l} |c_{2,i}|}
\end{equation} for $m$ large enough.
Furthermore, to ensure $\delta_1 \in A_m^{(1)}$, $\delta_2 \in A_m^{(2)}$, we
have to add the conditions $|\delta_1|, |\delta_2| < C' 2^{-m(1+\varepsilon)}$.
With $C_j = \max\{1, 3 \max_{i=3,\ldots,l} |c_{j,i}|\}$ for $j=1,2$, this
yields
\begin{align*}
 & \Lambda(\{I_{\{1,\ldots,l\}}(\bz +\delta), \ \delta_i \in A_m^{(i)}, \ i=1,\ldots,l\})\\
\geq{} & \Lambda\left(\left\{\left(t_{\delta_1\delta_2},\frac{z_1+\delta_1}{f(t_1-t_{\delta_1\delta_2})},f\right),
 \, |\delta_j|< \frac{C'}{C_j} 2^{-m(1+\varepsilon)}, \, j=1,2\right\}\right)\\
={} & \frac {(C')^2 2^{-2m(1+\varepsilon)} \cdot \Prob_F(\{f\})}{y_0^2 C_1 C_2 |z_1 f'(t_2-t_0) - z_2 f'(t_1-t_0)|} + o(2^{-2m(1+\varepsilon)})
\end{align*}
where we use the same argumentation as in the proof of Proposition 
\ref{intersectintens}. \qed
\end{proof}

\def\proofname{Proof of Proposition \ref{more-3points}}

\begin{proof}
We prove the assertion by conditioning on $Z(t_i)$ being in intervals of 
different size for each $i \in \{1,\ldots,l\}$ instead of $Z(t_i) \in A_m(z_i)$
for all $i=1,\ldots,l$. We choose these intervals such that some restrictions
on the intersection sets vanish asymptotically and we can resort to the results
on the intersection of two curves.

The calculations in the proof of Proposition \ref{3points} yield
$$|h_i(\delta_1,\delta_2)| \leq (|c_{1,i}|+o(1)) \cdot |\delta_1| 
+ (|c_{2,i}|+o(1)) \cdot |\delta_2| \leq 2^{-m} (|c_{1,i}| + |c_{2,i}| + o(1))$$
for $(\delta_1, \delta_2) \in A_m^{(1)} \times A_m^{(2)}$. Thus, for any 
$\varepsilon >0$, using the same arguments as in the proof of Lemma
\ref{lambda-bound}, we can replace $m$ by $\lfloor m(1-\varepsilon)\rfloor$ in
Equation \eqref{eq:delta-assessment} and get that 
$h_i(\delta_1,\delta_2) \in A_{\lfloor m (1-\varepsilon)\rfloor}^{(i)}$ holds
for
$|\delta_j| < \frac{C' 2^{-\lfloor m(1-\varepsilon)\rfloor (1+\varepsilon)}}{3 |c_{j,i}|}
 \sim 2^{\varepsilon^2m} 2^{-m}$, $j=1,2$, and for $m$ large enough.
Therefore, 
$$ h_i(\delta_1, \delta_2) \subset A_{\lfloor m(1-\varepsilon)\rfloor}^{(i)}, \quad i=3, \ldots, l$$
for all $\delta_1 \in A_m^{(1)} \subset (-2^{-m}, 2^{-m}]$, $\delta_2 \in A_m^{(2)} \subset (-2^{-m}, 2^{-m}]$
if $m$ is sufficiently large. With $A_{m,\varepsilon} = A_m^{(1)} \times A_m^{(2)} \times \times_{i=3}^l A_{\lfloor m(1-\varepsilon)\rfloor}^{(i)}$,
this implies
\begin{align*}
& \left\{I_{\{1,\ldots,l\}}(\bz + \delta),\ \delta \in A_{m,\varepsilon} \right\}\displaybreak[0]\\
={} &\left\{I_{\{1,\ldots,l\}}(\bz + \delta),\ (\delta_1,\delta_2)  \in A_m^{(1)} \times A_m^{(2)},\
   h_i(\delta_1,\delta_2) \in A_{\lfloor m(1-\varepsilon)\rfloor}^{(i)}, \ 3 \leq i \leq l\right\}\displaybreak[0]\\
={} & \left\{\left(t_{\delta_1\delta_2}, \frac{z_1+\delta_1}{f(t_1-t_{\delta_1\delta_2})}, f\right),\ \delta_1 \in A_m^{(1)},\ \delta_2 \in A_m^{(2)}\right\}
\end{align*}
and, therefore
\begin{align*}
 \Lambda\left(\left\{I^{(m)}_{\{1,\ldots,l\}}(\bz + \delta),\ \delta \in A_{m,\varepsilon} \right\}\right) 
={} & \frac{2^{-2m} y_0^{-2} \Prob_F(\{f\})}{|z_1 f'(t_2-t_0) - z_2 f'(t_1-t_0)|}
 + o(2^{-2m}).
\end{align*}
Let $A_{m,\varepsilon}(\bz) = A_m(z_1) \times A_m(z_2) \times
\times_{i=3}^l A_{\lfloor m(1-\varepsilon)\rfloor}(z_i)$.
By conditioning on $Z(\bt) \in A_{m,\varepsilon}(\bz)$, for
$I_{\{1,\ldots,l\}}(\bz) = \{ (s_1, u_1, f_1), \ldots, (s_k, u_k, f_k)\}$, $l\geq 3$,
we apply L\'evy's ``Upward'' Theorem and end up with
\begin{align*}
 & \Prob((t_0,y_0,f) \in \Pi \mid |\Pi \cap I_{\{1,\ldots,l\}}(Z(\bt))| = 1, Z(\bt)=\bz) \nonumber \\
={} & \lim_{m \to \infty} \Prob\Big(\Big|\Pi \cap \Big\{\left(t_{\delta_1,\delta_2}, y_{\delta},f\right):
    \ \delta \in A_{m,\varepsilon}\Big\}\Big| = 1 \ \Big|
     \nonumber \\
 & \hspace{4.5cm} \ |\Pi \cap I_{\{1,\ldots,l\}}(Z(\bt))| = 1,\ Z(\bt) \in A_{m,\varepsilon}(\bz)\Big) \nonumber\\
={} & \frac{y_0^{-2} \Prob_F(\{f\})}{|z_1 f'(t_2-t_0) - z_2 f'(t_1-t_0)|} \cdot 
\bigg(\sum_{j=1}^{k} \frac{u_j^{-2} \Prob_F(\{f_{j}\})}{|z_1 f'_{j}(t_2-s_j)-z_2 f'_{j}(t_1-s_j)|}\bigg)^{-1}.
\end{align*}
Note that L\'evy's ``Upward'' Theorem implies that, for $\Prob_{Z(\bt)}$-a.e.\ 
$\bz > \mathbf{0}$, the right-hand side of \eqref{eq:more-3points} does not
depend on the choice of the labelling. \qed
\end{proof}

\def\proofname{Proof of Proposition \ref{Kintens}}

\begin{proof}
First, we note that by renumbering it suffices to show the result for $i=1$.
The idea of this proof is to assess the set $D_1^{(f)}$ by the sets 
$D_{1,\cap{}}^{(m)}$ from below and $D_{1,\cup{}}^{(m)}$ from above. Here,
$D_{1,\cap{}}^{(m)}$ consists of all first components of $I^{(m)}_{\{1\}}(\bz)$
which are not part of any intersections $I_A^{(m)}(\bz)$, $A \supsetneq \{1\}$,
and  $D_{1,\cup{}}^{(m)}$ is the set of the first components of 
$\bigcup_{A \supset \{1\}} I^{(m)}_A(\bz)$. Analogously, 
$\Lambda\left(I_{\{i\}}^{(m)}(\bz) \cap S_f\right)$ can be bounded from below
and above by replacing $D_1^{(f)}$ in \eqref{eq:singlecurve} by 
$D_{1,\cap{}}^{(m)}$ and $D_{1,\cup{}}^{(m)}$, respectively. We show that the
difference, which consists of blurred intersections $I_A^{(m)}(\bz)$, 
$A \supsetneq \{1\}$, vanishes asymptotically.

Let $A_m^{(i)} = A_m(z_i) - z_i$.
Then, for any $\delta \in \times_{i=1}^n (-z_i, \infty)$ we define
\begin{align}
 D_{1,\delta}^{(f)} ={} & \left\{t \in \R:\ \left(t,\frac{z_1+\delta_1}{f(t_1-t)},f\right) 
\in I_{\{1\}}(\bz + \delta)\right\} \nonumber\\
={} &
\left\{t \in \R: \ \frac{z_1+\delta_1}{f(t_1-t)} < \bigwedge_{i=2}^n \frac{z_i+\delta_i}{f(t_i-t)} \right\}
\label{eq:def-d1delta}.
\end{align}
Thus, with 
$D_{1,\cap{}}^{(m)} = \bigcap_{\delta \in \times_{i=1}^n A_m^{(i)}} D_{1,\delta}^{(f)}$
and
$D_{1,\cup{}}^{(m)} = \bigcup_{\delta \in \times_{i=1}^n A_m^{(i)}} D_{1,\delta}^{(f)}$,
we get
\begin{equation*} \label{eq:nesting1}
D_{1,\cap}^{(m)} \subset D_1^{(f)} \subset D_{1,\cup}^{(m)}.
\end{equation*}

Furthermore, we have
\begin{align}
 & \left\{(t,y,f) \in S_f: \ t \in D_{1,\cap}^{(f)}, \ y f(t_1-t) \in A_m(z_1)\right\} \nonumber \displaybreak[0]\\
\subset I_{\{1\}}^{(m)}(\bz) \subset & \left\{(t,y,f) \in S_f: \ t \in D_{1,\cup}^{(f)}, \ y f(t_1-t) \in A_m(z_1)\right\}. \label{eq:nesting2}
\end{align}

Now, let $t \in D_{1,\cup}^{(m)} \setminus D_{1,\cap}^{(m)}$. Then, by 
definition of $D_{1,\cap}^{(f)}$ and $D_{1,\cup}^{(f)}$, there exist 
$\delta^{(1)}$, $\delta^{(2)} \in \times_{i=1}^n A_m^{(i)}$ such that 
$t \in D_{1,\delta^{(1)}}^{(f)}$, but $t \notin D_{1,\delta^{(2)}}^{(f)}$.
That is, by Equation \eqref{eq:def-d1delta},
\begin{align*}
\frac{z_1+\delta_1^{(1)}}{f(t_1-t)} {}<{} \bigwedge_{i=2}^n \frac{z_i + \delta_i^{(1)}}{f(t_i-t)} \quad \textrm{and}
 \quad \frac{z_1+\delta_1^{(2)}}{f(t_1-t)} {}\geq{} \bigwedge_{i=2}^n \frac{z_i + \delta_i^{(2)}}{f(t_i-t)}.
\end{align*}
By continuity, a $\delta \in \times_{i=1}^n A_m^{(i)}$ exists such
that
$\frac{z_1+\delta_1}{f(t_1-t)} {}={} \bigwedge_{i=2}^n \frac{z_i + \delta_i}{f(t_i-t)},$
i.e.\ $t \in T_1^{(m)} = \{t \in \R: (t,y,f) \in \bigcup_{A:\, \{1\} \subsetneq A} I_A^{(m)}(\bz) \textrm{ for some } y >0\}$.
Thus,
\begin{equation} \label{eq:blurred-diff}
 D_{1,\cup}^{(m)} \setminus D_{1,\cap}^{(m)} \subset T_1^{(m)}.
\end{equation}
By definition, $T_1^{(m)}$ denotes the set of first components involved in any
blurred intersection and we have 
$$ \textstyle T_1^{(m)} \searrow T_1 = 
\big\{t \in \R: (t,y,f) \in \bigcup_{A: \, \{1\} \subsetneq A} I_A(\bz)
\textrm{ for some } y >0\big\}, \quad m \to \infty,$$
and $T_1$ is finite by Assumption \eqref{eq:intersection}. Therefore, dominated
convergence yields 
\begin{equation} \label{eq:blurred-diff-van}
 \textstyle \int_{T_1^{(m)}} f(t_1-t) \sd t \searrow 0, \quad m \to \infty.
\end{equation}

Thus, by Equations \eqref{eq:nesting2} and \eqref{eq:blurred-diff} we get
\begin{align*}
& \Lambda\big(I_{\{1\}}^{(m)}(\bz)  \ \Delta \ \big\{(t,y,f) \in S_f: \
t \in D_1^{(f)}, \ y f(t_1-t) \in A_m(z_1)\big\}\big)\\
\leq{} & \Lambda\big(\big\{(t,y,f) \in S_f: \ t \in D_{1,\cup}^{(f)}, \ y f(t_1-t) \in A_m(z_1)\big\} \ \big\backslash\\
& \hspace{0.5cm} \big\{(t,y,f) \in S_f: \ t \in D_{1,\cap}^{(f)}, \ y f(t_1-t) \in A_m(z_1)\big\}\big) \displaybreak[0]\\
\leq{} & \textstyle \Lambda\big(\big\{(t,y,f) \in S_f: \ t \in T_1^{(m)},\ yf(t_1-t) \in A_m(z_1)\big\}\big)\\
={} & \textstyle \Prob_F(\{f\})  \int_{A_m^{(1)}} \int_{T_1^{(m)}} \frac{f(t_1-t)}{(z_1+\delta_1)^2} \sd t \sd \delta_1
  = \int_{A_m^{(1)}} \frac{o(1)}{(z_1+\delta_1)^2} \sd \delta_1 \in o(2^{-m}).
\end{align*}
The last equality follows from Equation \eqref{eq:blurred-diff-van}.
Hence, we have
\begin{align*}
\textstyle \Lambda(I_{\{1\}}^{(m)}(\bz)) ={} & \textstyle \Prob_F(\{f\}) \int_{D_1^{(f)}} \int_{A_m^{(1)}} \frac{f(t_1-t)}{(z_1+\delta_1)^2} \sd \delta_1 \sd t  + o(2^{-m}) \displaybreak[0]\\
 ={} & \textstyle 2^{-m} \cdot \Prob_F(\{f\}) \cdot \int_{D_1^{(f)}} \frac{f(t_1-t)}{z_1^2} \sd t + o(2^{-m})
\end{align*}
which completes the proof. \qed
\end{proof}

\def\proofname{Proof of Lemma \ref{neglectprobs}}

\begin{proof}
We prove that condition \eqref{eq:problems} has probability 0 for all fixed
index sets $A_1, A_2 \subset \{1,\ldots,n\}$. By renumbering, we may assume
that $A_1=\{1,\ldots,r\}$ and $A_2=\{q, \ldots, q +s-1\}$ with $q \leq r$.

Assume that $\Prob(Z(\bt) \textrm{ satisfies } \eqref{eq:problems}) > 0.$
In a first step we only consider those realizations of $Z(t_1), \ldots, Z(t_r)$
with $J_{A_1}(Z(t_1),\ldots,Z(t_r)) \neq \emptyset$.
Then, by the calculations in Propositions \ref{intersectintens}, \ref{3points}
and \ref{Kintens}, we get that
$$\Prob(|\Pi \cap I_{A_1}^{(m)}(Z(t_1),\ldots,Z(t_r))| = 1) \notin \OO(2^{-2m(1+\varepsilon)})$$
 for any $\varepsilon > 0$ and 
 $\Prob(|\Pi \cap I_{B_j}^{(m)}(Z(t_1),\ldots,Z(t_r))| = 1, \ j=1,\ldots,k) \in \OO(2^{-3m})$
for any pairwise disjoint $B_1, \ldots, B_k$ with 
$\bigcup_{j=1}^k B_j = A_1$, $k \geq2$. This yields 
$|\Pi \cap J_{A_1}(Z(t_1),\ldots,Z(t_r))| = 1$ almost surely.

Similarly, we have
$|\Pi \cap J_{A_2}(Z(t_q),\ldots,Z(t_{q+s-1}))| = 1$
for almost every $Z_q, \ldots, Z_{q+s-1}$ with
$J_{A_2}(Z(t_q),\ldots,Z(t_{q+s-1})) \neq \emptyset$.
As
\begin{align*}
&\textstyle \{\omega: \ Z(\bt) \textrm{ satisfies } \eqref{eq:problems}\}\\
 \textstyle \subset{} & \textstyle \{\omega: \ J_{A_1}(Z(t_1), \ldots, Z(t_r)) \neq \emptyset\}
    \cap  \{\omega: \ J_{A_2}(Z(t_q), \ldots, Z(t_{q+s-1})) \neq \emptyset\},
\end{align*}
we have $|\Pi \cap J_{A_1}(Z(t_1),\ldots,Z(t_{r}))| = |\Pi\cap J_{A_2}(Z(t_q),\ldots,Z(t_{q+s-1}))| = 1$
for $Z(\bt)$ satisfying \eqref{eq:problems} almost surely.
Therefore,
we get $\Prob(|\Pi \cap K_{(t_i,Z(t_i))}| \geq 2) > 0$ for every
$i \in A_1 \cap A_2$ since $J_{A_1 \cup A_2}(Z(\bt)) = \emptyset$. 
This is a contradiction to Corollary \ref{as-onepoint}. \qed
\end{proof}

\begin{acknowledgements}
The research of M.~Oesting was supported by the German Research Foundation DFG
through the Graduiertenkolleg 1023 \emph{Identification in Mathematical Models:
Synergy of Stochastic and Numerical Methods}, Universit\"at G\"ottingen, in 
form of a scholarship. Both authors have also been financially supported partly
by Volkswagen Stiftung within the project `Mesoscale Weather Extremes -- Theory,
Spatial Modeling and Prediction (WEX-MOP)'.
They are grateful to two anonymous referees for numerous valuable suggestions
improving this article. The authors also thank Thomas Rippl for helpful 
discussions on regular conditional probabilities and martingales.
\end{acknowledgements}

\bibliographystyle{spbasic}      
\bibliography{ref}   

\end{document}